\documentclass[final,onefignum,onetabnum]{siamart190516}

\usepackage{braket,amsfonts}

\usepackage{array}

\usepackage[caption=false]{subfig}
\usepackage{pgfplots}

\newsiamthm{claim}{Claim}
\newsiamremark{remark}{Remark}
\newsiamremark{hypothesis}{Hypothesis}
\crefname{hypothesis}{Hypothesis}{Hypotheses}

\usepackage{algorithmic}

\usepackage{graphicx,epstopdf}

\Crefname{ALC@unique}{Line}{Lines}

\usepackage{amsopn}

\usepackage{xspace}
\usepackage{bold-extra}
\usepackage[most]{tcolorbox}

\colorlet{texcscolor}{blue!50!black}
\colorlet{texemcolor}{red!70!black}
\colorlet{texpreamble}{red!70!black}
\colorlet{codebackground}{black!25!white!25}

\lstdefinestyle{siamlatex}{%
	style=tcblatex,
	texcsstyle=*\color{texcscolor},
	texcsstyle=[2]\color{texemcolor},
	keywordstyle=[2]\color{texemcolor},
	moretexcs={cref,Cref,maketitle,mathcal,text,headers,email,url},
}

\tcbset{%
	colframe=black!75!white!75,
	coltitle=white,
	colback=codebackground, 
	colbacklower=white, 
	fonttitle=\bfseries,
	arc=0pt,outer arc=0pt,
	top=1pt,bottom=1pt,left=1mm,right=1mm,middle=1mm,boxsep=1mm,
	leftrule=0.3mm,rightrule=0.3mm,toprule=0.3mm,bottomrule=0.3mm,
	listing options={style=siamlatex}
}

\newtcblisting[use counter=example]{example}[2][]{%
	title={Example~\thetcbcounter: #2},#1}

\newtcbinputlisting[use counter=example]{\examplefile}[3][]{%
	title={Example~\thetcbcounter: #2},listing file={#3},#1}

\DeclareTotalTCBox{\code}{ v O{} }
{ 
	fontupper=\ttfamily\color{black},
	nobeforeafter,
	tcbox raise base,
	colback=codebackground,colframe=white,
	top=0pt,bottom=0pt,left=0mm,right=0mm,
	leftrule=0pt,rightrule=0pt,toprule=0mm,bottomrule=0mm,
	boxsep=0.5mm,
	#2}{#1}

\patchcmd\newpage{\vfil}{}{}{}
\usepackage{color}
\newcommand{\nn}{\nonumber}
\usepackage{enumerate}

\newcommand{\abs}[1]{\lvert#1\rvert}

\newcommand{\dual}[2]{\left\langle\,#1,#2\right\rangle}
\newcommand{\Lr}[1]{\left(#1\right)}

\newcommand{\sset}[2]{\{\,#1\,\mid\,#2\,\}}

\newcommand{\nm}[2]{\|\,#1\,\|_{#2}}
\newcommand{\jump}[1]{[\![#1]\!]}
\newcommand{\aver}[1]{\{\!\{#1\}\!\}}
\newcommand{\wnm}[1]{|\!|\!|#1|\!|\!|}

\newcommand{\mc}[1]{\mathcal{#1}}
\newcommand{\mb}[1]{\mathbb{#1}}

\newcommand{\wt}[1]{\widetilde{#1}}

\def\a{a^\eps}
\def\b{b^\eps}
\def\A{\mathcal{A}}

\def\al{\alpha}

\def\eps{\varepsilon}
\def\na{\nabla}
\def\pa{\partial}
\def\lam{\lambda}
\def\Lam{\Lambda}

\def\Om{\Omega}

\def\x{\times}
\def\dx{\,\mathrm{d}x}

\def\ds{\,\mathrm{d}s}
\def\n{\mathbf{n}}
\def\A{\mathcal{A}}
\def\T{\mathcal{T}}

\def\E{\mathcal{E}}
\def\N{\mathcal{N}}
\def\R{\mathbb{R}}

\def\hmm{\textsc{HMM}}

\definecolor{dblue}{rgb}{0, .1, .6}

\begin{tcbverbatimwrite}{tmp_\jobname_header.tex}
	\title{A Nitsche Hybrid multiscale method with non-matching grids
	}
	
	\author{Pingbing Ming\thanks{LSEC, Institute of Computational Mathematics and Scientific/Engineering Computing, AMSS,
			Chinese Academy of Sciences, No. 55, East Road Zhong-Guan-Cun, Beijing 100190, China
			and School of Mathematical Sciences, University of Chinese Academy of Sciences, Beijing 100049, China (\email{mpb@lsec.cc.ac.cn},\email{songsq@lsec.cc.ac.cn}).}
		\and SIQI SONG\footnotemark[2]}
	
	\headers{A Nitsche Hybrid multiscale method with non-matching grids}{Pingbing Ming, and Siqi Song}
\end{tcbverbatimwrite}
	\title{A Nitsche Hybrid multiscale method with non-matching grids
	}
	
	\author{Pingbing Ming\thanks{LSEC, Institute of Computational Mathematics and Scientific/Engineering Computing, AMSS,
			Chinese Academy of Sciences, No. 55, East Road Zhong-Guan-Cun, Beijing 100190, China
			and School of Mathematical Sciences, University of Chinese Academy of Sciences, Beijing 100049, China (\email{mpb@lsec.cc.ac.cn},\email{songsq@lsec.cc.ac.cn}).}
		\and SIQI SONG\footnotemark[2]}
	
	\headers{A Nitsche Hybrid multiscale method with non-matching grids}{Pingbing Ming, and Siqi Song}

\ifpdf
\hypersetup{ pdftitle={A Nitsche Hybrid multiscale method with non-matching grids} }
\fi

\begin{document}
	\maketitle
	
	\begin{tcbverbatimwrite}{tmp_\jobname_abstract.tex}
		\begin{abstract}
			We propose a Nitsche method for multiscale partial differential equations, which retrieves the macroscopic information and the local microscopic information at one stroke. We prove the convergence of the method for second order elliptic problem with bounded and measurable coefficients. The rate of convergence may be derived for coefficients with further structures such as periodicity and ergodicity. Extensive numerical results confirm the theoretical predictions.
		\end{abstract}
		

	\end{tcbverbatimwrite}
			\begin{abstract}
			We propose a Nitsche method for multiscale partial differential equations, which retrieves the macroscopic information and the local microscopic information at one stroke. We prove the convergence of the method for second order elliptic problem with bounded and measurable coefficients. The rate of convergence may be derived for coefficients with further structures such as periodicity and ergodicity. Extensive numerical results confirm the theoretical predictions.
		\end{abstract}
		


	%

\section{Introduction}
Consider the elliptic problem with Dirichlet boundary condition
\begin{equation}\label{eq:prob}
	\left\{\begin{aligned}
		-\text{div}\bigl(\a(x)\nabla u^{\epsilon}(x)\bigr)&=f(x),\qquad  &&x\in D\subset \mathbb{R}^n,\\
		u^{\epsilon}(x)&=0, \qquad &&x\in\pa D,
	\end{aligned}
	\right.
\end{equation}
where $D$ is a bounded domain in $\mb{R}^n$ with $n=2,3$, and $\eps$ is a small parameter that signifies the multiscale nature of the problem. Problem~\eqref{eq:prob} may be viewed as a prototypical model of many multiscale procblems arising from a variety of contexts, such as the heat conduction and the electromagnetism in composites, or the transport of the porous media. The main quantities of interest for Problem~\eqref{eq:prob} are the macroscopical behavior of the solution and the local microscopical information of the solution~\cite{E:2011, Babuska:2014}. Many numerical methods have been developed in the literature to capture either the macroscopical behaviors or the microscopical information of the solution, such as the heterogeneous multiscale methods (HMM)~\cite{EEnquist:2003}, the multiscale finite element methods~\cite{HouWu:1997} and many others.

There are also some methods that aim to retrieve the coarse scale information and the local fine scale information
simultaneously. Such methods may be roughly grouped into three classes. The first one is the global-local method, which was originally proposed in~\cite{OdenVemaganti:2000, OdenVemaganti:2001}. The main idea is to solve the coarse scale problem by a numerical upscaling method firstly, and then solve the local problem around the defects or the places for which the fine scale information is of interest, while the coarse scale information is employed as the constraints. This idea has been incorporated into the HMM framework in~\cite{EEnquist:2003} and the performance has been thoroughly analyzed in~\cite{EMingZhang:2005} and~\cite{BabuskaLipton:2011}. The global-local method has been extended to solve an elastodynamical wave equation in~\cite{Babuska:2014}.

The second method is based on the domain decomposition idea, which has been exploited to solve the multiscale PDEs
in~\cite{Glowinski:03, Pironneau:2007, DengWu:2014, AbdulleJecker:2015, XuMing:2016}. The most relevant is the one in~\cite{AbdulleJecker:2015}. The authors therein used a discontinuous Galerkin HMM in a region with scale separation, while use a finite element method in a region without scale separation. The unknown boundary condition has been supplied by minimizing the difference between the solutions in the overlapped region. The well-posedness and the convergence of this method have been studied in~\cite{AbdulleJecker:2015} for the periodic media.

The third method relies on the hybridization idea~\cite{HuangLuMing:2018}. One solves the following variational problem:
Find $v_h\in X_h$ such that
\begin{equation}\label{eq:hyproblem}
	\dual{b^\eps\na v_h}{\na w}=\dual{f}{w}\qquad \text{for all\quad} w\in X_h,
\end{equation}
where $X_h$ is any finite element space, and we denote the $L^2(D)$ inner product by $\dual{\cdot}{\cdot}$.  Here  
\(
\b(x){:}=\rho(x)\a(x)+(1-\rho(x))\A(x),\)
where $\A$ is the effective matrix arising from the homogenization problem:
\begin{equation}\label{eq:homoprob}
	\left\{\begin{aligned}
		-\text{div}(\mathcal{A}(x) \na u_0)&=f(x),\qquad  &&x\in D,\\
		u_0&=0, \qquad &&x\in\pa D.
	\end{aligned}
	\right.
\end{equation}
The coefficient $\b$ is a hybridization of the microscopical coefficient and the macroscopical coefficient with a transition function $\rho$. Roughly speaking, the transition function takes one in the defected region and zero otherwise. The authors proved the well-posedness and the convergence of~\eqref{eq:hyproblem} for the bounded and measurable coefficient $\a$. The rate of convergence was derived for the periodic media and the quasi-periodic media. Numerical results in~\cite{HuangLuMing:2018} show that this hybrid method is comparable with the classical global-local method in terms of both the accuracy and the efficiency, while it is particularly suitable for the scenario that the microscopic coefficient $\a$ is only available in part of the domain, while outside this region, the coarse scale information is available for the coefficient fields.

The present work is a follow-up of~\cite{HuangLuMing:2018}, and there are two contributions. Firstly we employ the variational formulation of Nitsche~\cite{Nitsche:1971} to solve~\eqref{eq:hyproblem}, which allows for non-matching grid across the interface. Such numerical interface is caused by the local support of the transition function. The authors in~\cite{HuangLuMing:2018} employed the linear finite element over a body-fitted mesh to solve~\eqref{eq:hyproblem}. Highly refined mesh has to be used around the defect region to ensure the conformity of the mesh and the resolution of the local defects. From this aspect of view, the non-matching grid is more flexible in implementation. Indeed, as demonstrated in \S~\ref{sec:numerics}, fewer global degrees of freedom is required to achieve the desired accuracy compared to the original hybrid method. We note that Nitsche's method is a powerful tool to deal with the interface problem in finite element method and the discontinuous Galerkin method; See, e.g., ~\cite{BeckerHansboStenberg:2003,DengWu:2014,SongDengWu:2016,BurmanElfversonHansbo:2019}. Another contribution is a general method to construct the transition function, which is an essential ingredient of the hybrid method while seems missing in~\cite{HuangLuMing:2018}, because only the square defects have been dealt with therein, for which the transition function is a tensor product of a spline function in one dimension. It is nontrivial to find such explicit expression of $\rho$ for defects with irregular shape. 
Once a general transition function is constructed, it is straightforward to handle the defects with irregular shapes, and numerical results show that the method works well for such irregular defects without occurring extra cost.

To analyze the Nitsche hybrid method, we need a well-defined trace over the element boundary, which demands that $u_0\in H^{1+s}(D)$ with $s>1/2$. For smooth solution, the Nitsche hybrid method may be analyzed by combining the technique in~\cite{HuangLuMing:2018} and the standard way for analyzing DG method~\cite{Arnold:2001}. Unfortunately, such smoothness assumption on $u_0$ may not be true for a rough coefficient matrix $\mc{A}$, or a point load function $f$, or a nonconvex domain $D$. Hence we adapt the medius analysis~\cite{Gudi:2010_New} to the present problem. To deal with the non-matching grid that is not covered by the standard medius analysis~\cite{Gudi:2010_New,Gudi:2010_nonstandard,LuthenJuntunenStenberg:2018}, we construct a new enriching operator that measures the difference between the discontinuous finite element space and the Sobolev space $H^1$ over such triangulation. Such difference may be bounded by the jump of the function across the interface, which is independent of the mesh ratio. The enriching operator stems from~\cite{Brenner:1996} and~\cite{Karakshian:2003}, which plays an important role
in analyzing DG method~\cite{Karakshian:2003,Gudi:2010_New}, nonconforming finite element method~\cite{Brenner:1996,LiMingWang:2021}
and the virtual element method~\cite{BrennerSung:2019}, where we just name a few of them and refer to~\cite{BrennerSung:2019}
for an updated review. The main ingredient of the construction is the mesh
ratio dependent weights~\cite{Stenberg:1998,HeinrichPietsch:2002} instead of the standard arithmetic mean~\cite{Karakshian:2003}.
Using this enriching operator, we may prove the error estimate without any
regularity assumption on $u_0$. Though the error bounds weakly depend on the
mesh ratio, we may remove such dependence by adjusting the penalized parameter in Nitsche's variational formulation. Besides the non-matching grids,
the bounded measurable coefficient $a^\eps$ adds certain new difficulties.

The rest of the paper is organized as follows. We introduce the method in \S~\ref{sec:method}. The well-posedness and the error estimates of the proposed method are proved in \S~\ref{sec:analysis}, this is also the main theoretical result of the present work. We prove the main technical lemmas in \S~\ref{sec:pfaux}. Numerical examples for defects with various shapes are reported in \S~\ref{sec:numerics}. Some technical results are included in the Appendix.

Throughout this paper, we shall use Sobolev spaces $W^{r,p}(D)$  with norm $\nm{\cdot}{r,p,D}$ and semi-norm $\abs{\cdot}_{r,p,D}$, and we shall drop the  subscript $p$ when $p=2$. We refer to~\cite{AdamsFournier:2003} for details. We shall use $C$ as a generic constant independent of $\eps$, the mesh size $h$, $H$ and $H/h$, which may change from line to line.
\section{The Nitsche Hybrid Method}\label{sec:method}
To introduce the method, we fix some notations. Let $K_0$ be the defected region, and we slightly extend $K_0$ to $K_1$ and define $d{:}=\text{dist}(K_0,K_1)$. Denote $K_2=D\setminus K_1$, and let $\abs{K_i}{:}=\text{mes}K_i$ with $i=1,2$ and $\Gamma{:}=\pa {K_1}\setminus\pa D$. We construct a transition function $\rho$ satisfying
\[
\left\{\begin{aligned}
	K_1&=\text{supp}\;\rho,\qquad &&0\le\rho \le 1,\\
	\rho(x)&\equiv 1\quad&&\text{for\;}x\in K_0.
\end{aligned}\right.
\]
To this end, we firstly set $\rho\equiv 1$ in $K_0$ and $\rho\equiv 0$ outside $K_1$. A one layer mesh is constructed between $K_0$ and $K_1$. Secondly, we use a linear Lagrange interpolant over this triangulation to generate the transition function $\rho$ in $K_1\setminus K_0$, which is a continuous function. The triangulation between $K_0$ and $K_1$ for three different kinds of defects is plotted in Fig.~\ref{fig:rho}.
\begin{figure}[htbp]
	\centering
	{\includegraphics[width=3.7cm]{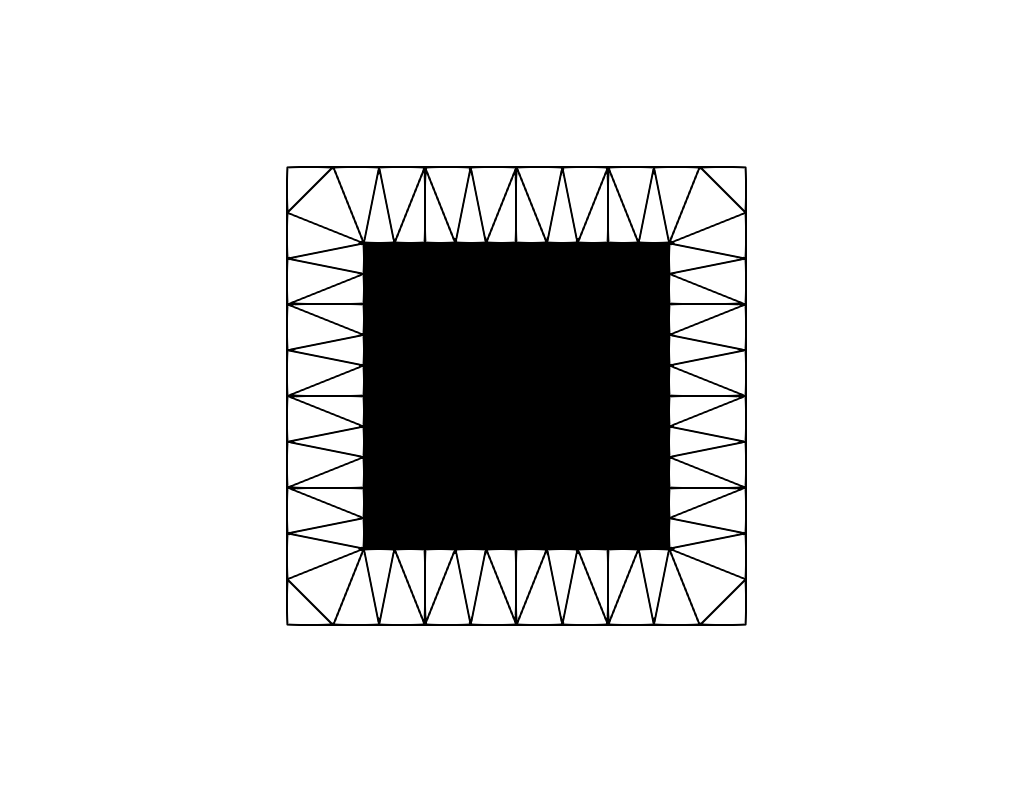}}
	{\includegraphics[width=3.7cm]{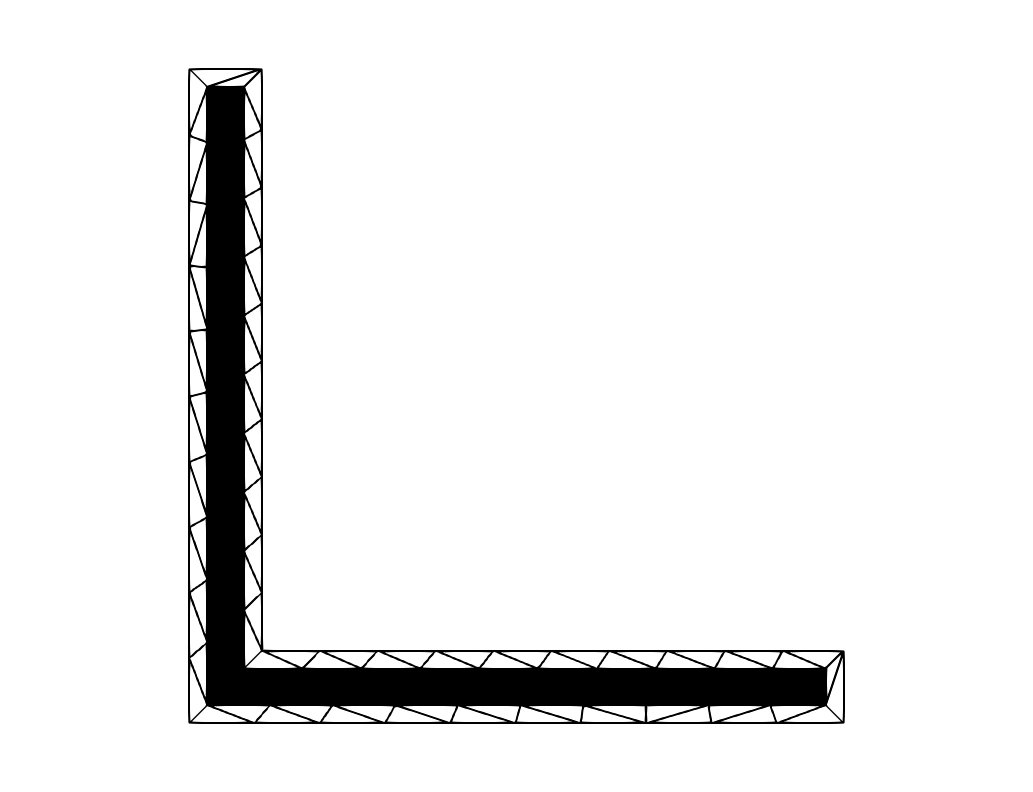}}
	{\includegraphics[width=3.7cm]{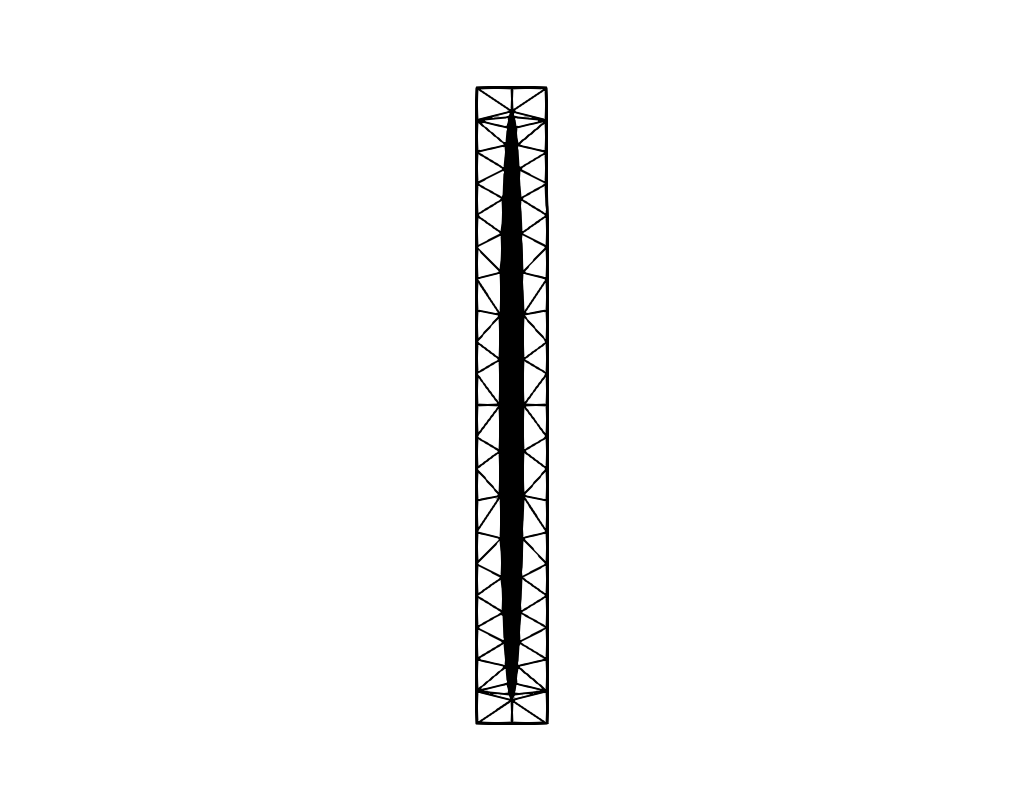}}
	\caption{One layer triangulation of (a) well defect; (b) channel defect; (c) Ellipse defect.}
	\label{fig:rho}
\end{figure}
This construction ensures that $\rho(x)=0$ for $x\in\Gamma$. In what follows, we do not assume any further smoothness on $\rho$. The effect of the smoothness on $\rho$ will be studied in \S~\ref{sect:rho}.

We triangulate $K_1$ and $K_2$ by $\T_h$ and $\T_H$ with the maximum meshsize $h$ and $H$, respectively. Hence we have a global triangulation $\T_{h,H}=\T_h\cup\T_H$ over the whole domain though $\T_h$ and $\T_H$ may not be conforming; See Fig.~\ref{fig:regionmesh} for an illustration of $\mc{T}_{h,H}$. We assume that both $\mc{T}_h$ and $\mc{T}_H$ are shape-regular in the sense of Ciarlet-Raviart~\cite{Ciarlet:1978} with the chunkiness parameter $\sigma$.
\begin{figure}[tbhp]
	\centering
	{\includegraphics[width=5.8cm]{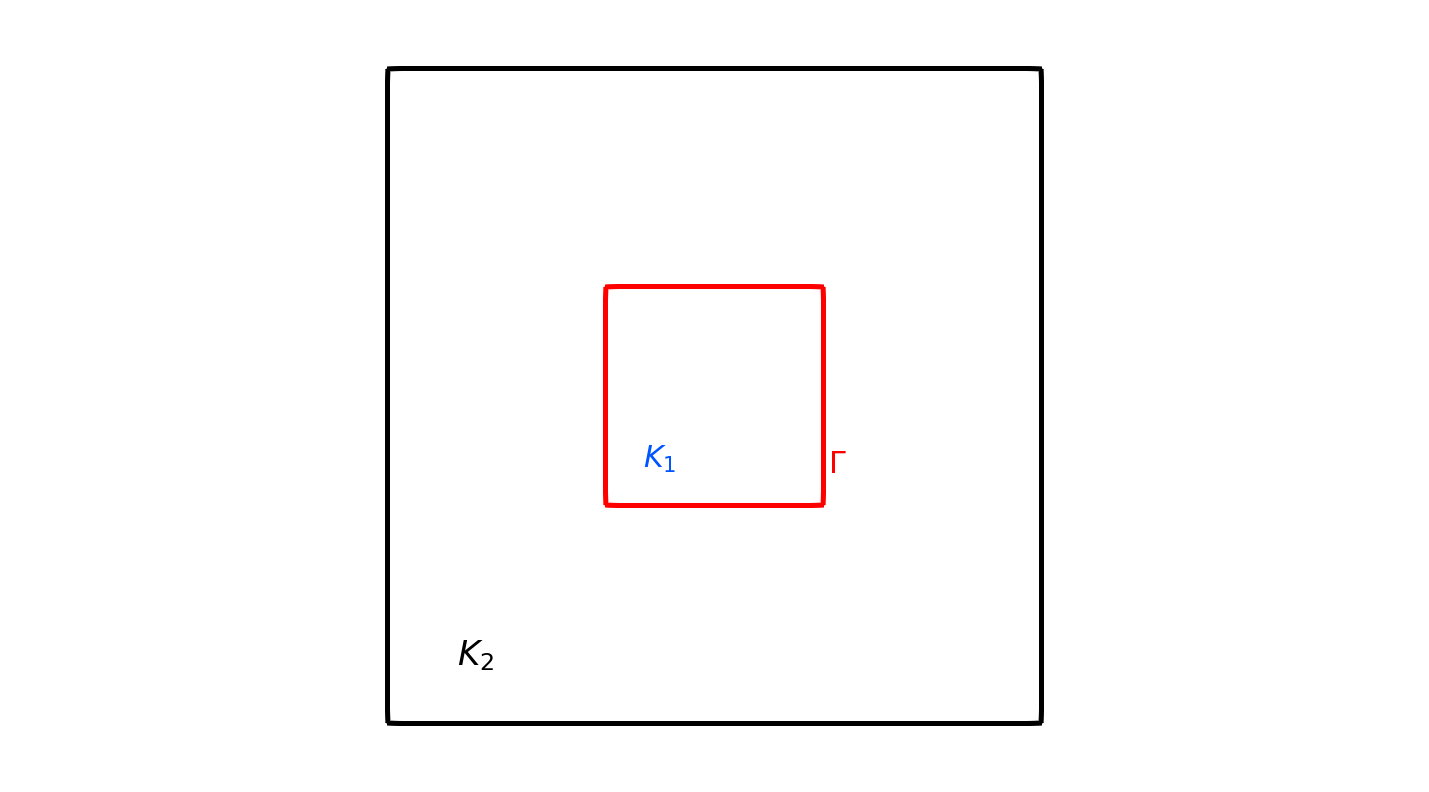}}
	{\includegraphics[width=5.8cm]{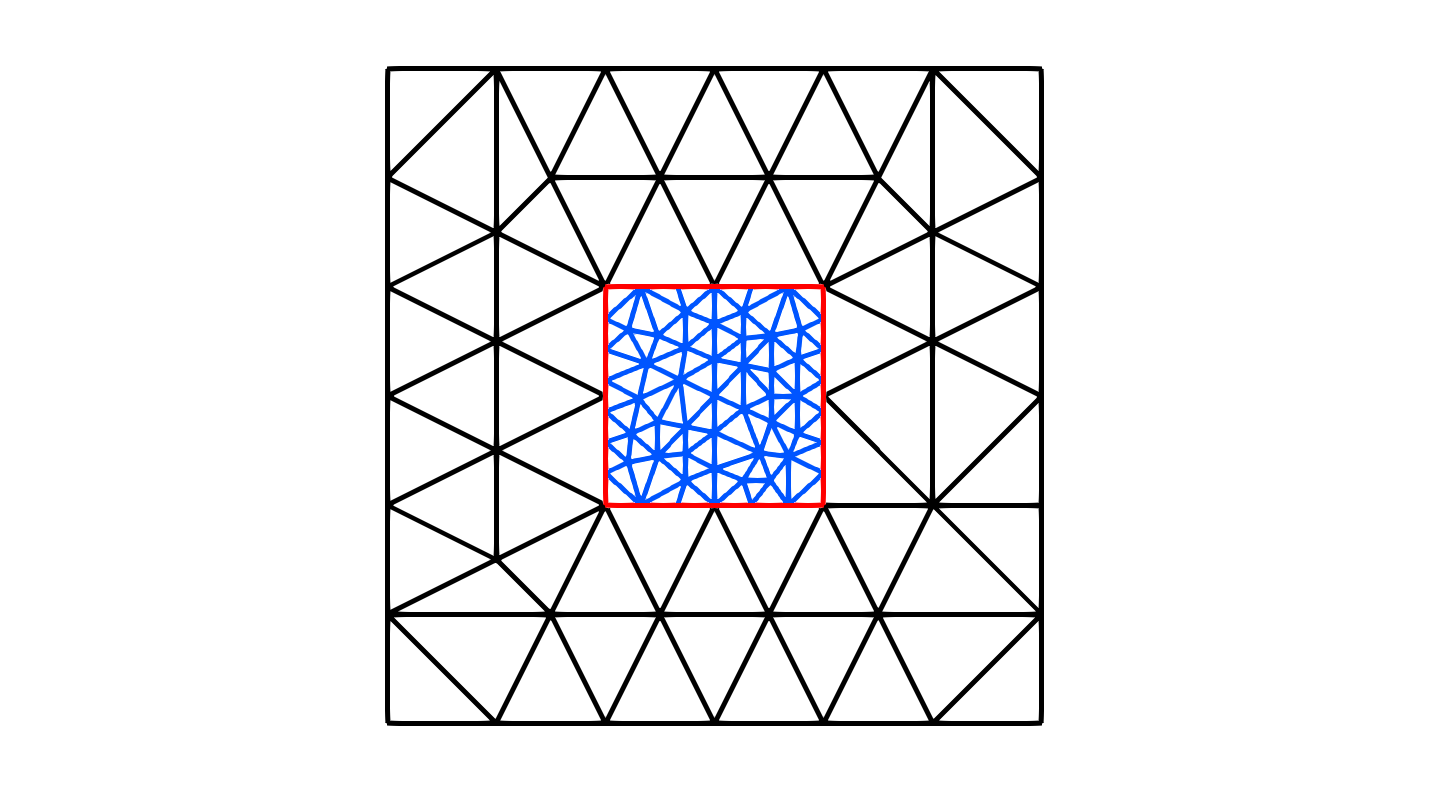}}
	\caption{Left: the region $K_1$ and $K_2$; Right: The mesh $\T_h$ in $K_1$ (blue), $\T_H$ in $K_2$ (black) and the interface $\Gamma$ (red).}
	\label{fig:regionmesh}
\end{figure}

Denote the set of all edges (faces in $n=3$) of $\T_h$ and $\T_H$ on $\Gamma$ by $\E_h$ and $\E_H$, respectively. Moreover, we denote by $\E_{\cap}$ the boundary mesh obtained by intersecting $\E_h$ and $\E_H$, i.e., $\E_{\cap}=\{e_\cap=e\cap E: e\in\E_h, E\in\E_H \}$. For convenience, we define
\[
\T_h^\Gamma{:}=\{\tau\in \mathcal{T}_{h}: \overline{\tau}\cap\Gamma\neq\emptyset\}\quad\text{and\quad}
\T_H^\Gamma{:}=\{\tau\in\mathcal{T}_{H}: \overline{\tau}\cap\Gamma\neq\emptyset \},
\]
and $\T_{h,H}^\Gamma{:}=\T_h^\Gamma\cup\T_H^\Gamma$. We refer to Fig.~\ref{fig:Gamma} for a plot of the mesh around the interface. 	
\begin{figure}[h]
	\centering
	\includegraphics{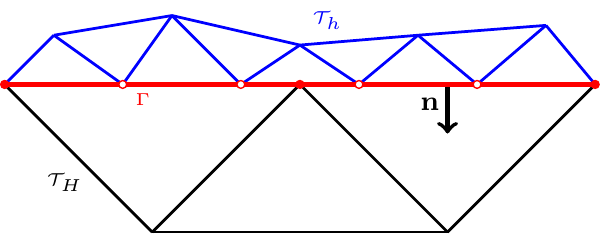}
	\caption{The collection $\T_h^\Gamma$ (blue) and $\T_H^\Gamma$ (black).}
	\label{fig:Gamma}
\end{figure}

Let $X_h$ and $X_H$ be the Lagrange finite element spaces consisting of piecewise polynomials of degree $r$ over $\mathcal{T}_h$ and $\mathcal{T}_H$, respectively. Over $\mathcal{T}_{h,H}$, we define
\begin{equation}\label{eq:space}
	X_{h,H}{:}=X_h\times X_H\quad \text{and}\quad
	X_{h,H}^0{:}=\{v\in X_{h,H}: v\arrowvert_{\pa D}=0\}.
\end{equation}

For each $e\in\E_\cap$,  there exists $(\tau_1,\tau_2)\in\T_h^\Gamma\times\T_H^\Gamma$ such that $\bar{\tau}_1\cap \bar{\tau}_2=e$. Define $h_e{:}=h_{\tau_1}$ and $H_e{:}=h_{\tau_2}$, and the weighted average and jump for $v=(v_1,v_2)\in X_h\times X_H$  on $e$  are defined by
\[
\aver{v}_\omega{:}=\omega_1v_1+\omega_2v_2\quad\text{and}\quad \jump{v}{:}=v_1-v_2
\]
with the weights
\[
\omega_1=\frac{h_e}{h_e+H_e}\quad \text{and} \quad  \omega_2=\frac{H_e}{h_e+H_e}.
\]
Similar to the magic DG formula in~\cite{Arnold:2001}, we have
\begin{equation}\label{eq:magic}
	\aver{v}_\omega\jump{w}-\jump{vw}=-\jump{v}\aver{w}^\omega,
\end{equation}
where $\aver{v}^\omega{:}=\omega_2v_1+\omega_1v_2$.

The bilinear form is defined for any $v,w\in X_{h,H}$ as
\begin{align*}
	B^\eps_h(v,w){:}&=\sum_{i=1}^2\int_{K_i}\b_h\na v\na w\dx
	-\sum_{e\in \E_\cap}\int_e\aver{\b_h\na v\cdot \n}_\omega\jump{w}\ds\\
	&\quad-\sum_{e\in \E_\cap}\int_e\aver{(\b_h)^{T}\na w\cdot \n}_\omega\jump{v}\ds+\sum_{e\in\E_\cap}\int_e\frac{\gamma}{H_e+h_e}\jump{v}\jump{w}\ds,
\end{align*}
where $\n$ is the unit outward normal vector of $\Gamma$ from $K_1$ to $K_2$, and $\gamma$ is the penalized parameter. Here
\begin{equation}\label{eq:aphyco}
	\b_h(x){:}=\rho(x)\a(x)+(1-\rho(x))\A_h(x)
\end{equation}
with $\A_h$ an approximation of $\A$. To step further, we assume that $\a$ belongs to a set $\mathcal{M}(\lam,\Lam;D)$  defined by
\begin{equation*}
	\begin{split}
		\mathcal{M}(\lam,\Lam;D){:}=\{a\in [L^\infty(D)]^{n\x n}\ :& \ \xi\cdot a(x)\xi\geq\lam\abs{\xi}^2,\ \xi\cdot a(x)\xi\geq(1/\Lam)\abs{a(x)\xi}^2\\
		&\text{for any~}\xi\in\mathbb{R}^n \text{~and~}a.e. \ x \text{~in~} D\},	
	\end{split}
\end{equation*}
where $\abs{\cdot}$ denotes the Euclidean norm in $\mathbb{R}^n$. By the theory of H-convergence~\cite{Tartar:2009}, we have $\A\in\mc{M}(\lam,\Lam;D)$. For any reasonable approximation $\mc{A}_h$, we may assume that the hybrid coefficient $\b_h\in\mc{M}(\lam',\Lam';D)$ for certain positive constants $\lam'$ and $\Lam'$. For example, if we use HMM~\cite{HMMreview1,HMMreview2, EMingZhang:2005,YueE:2007} to compute the effective matrix, then $\A_h\in\mc{M}(\lam,\Lam;D)$. Hence $\b_h\in\mc{M}(\lam,\Lam;D)$. If we use some other numerical upscaling methods, e.g.,~\cite{Gloria:2011, LiMingTang:2012, Gloria:2016, HuangMingSong:2020}, then $\b_h\in\mc{M}(\lam',\Lam';D)$ with certain constants $\lam'$ and $\Lam'$, which depend on $\lam$ and $\Lam$, but not exactly the same. To quantity the approximation error for the effective matrix, we define $e(\textsc{HMM}){:}=\max_{x\in \bar{K}_2}\nm{(\A-\A_h)(x)}{F}$, where $\nm{\cdot}{F}$ is the Frobenius norm of a matrix.

The approximation problem is defined as: Find $v_h\in X_{h,H}^0$ such that
\begin{equation}\label{eq:disvar}
	B^\eps_h(v_h,w)=\dual{f}{w}\quad \text{for all\quad} w\in X_{h,H}^0.
\end{equation}
\section{Error Estimate for the Nitsche Hybrid Method}\label{sec:analysis}
\subsection{Accuracy for retrieving the macroscopic information}
In this part, we estimate the error between the hybrid solution and the homogenized solution. For any $v\in X_{h,H}^0$, we define the broken energy norm as
\begin{equation}\label{eq:norm}
	\wnm{v}:=\left(\sum_{i=1}^2\abs{ v}_{1,K_i}^2+\sum_{e\in \E_\cap}\frac{\gamma}{H_e+h_e}\nm{\jump{v}}{0,e}^2\right)^{1/2}.
\end{equation}
The following lemma gives the continuity and coercivity of the bilinear form $B^\eps_h$ with respect to the above broken energy norm.
\begin{lemma}\label{lemma:coer}
Assume that $\b_h$  is in $\mathcal{M}(\lam',\Lam'; D)$.  Let $\gamma_0=8\Lam'^2C_{\text{inv}}^2/\min(1,\lam'^2)$, with $C_{\text{inv}}$ in~\eqref{eq:traceinv}.  If $\gamma\geq\gamma_0$, for all $v,w\in X_{h,H}$, then
	\begin{align}
		\abs{B^\eps_h(v,w)}&\le 2\max(1,\Lam')\wnm{v}\wnm{w},\label{eq:cont}\\
		B^\eps_h(v,v)&\ge\min\Lr{1/2,\lam'/2}\wnm{v}^2,\label{eq:coer}
	\end{align}
	where the constant $C_{\text{inv}}$ depends only on $r$ and $\sigma$ such that
	\begin{equation}\label{eq:traceinv}
		\nm{v}{0,e}\le C_{\text{inv}}\abs{h_\tau}^{-1/2}\nm{v}{0,\tau}\quad \text{for all\quad}v\in P_{r}(\tau),\;e\subset\pa\tau,\tau\in\mc{T}_{h,H}.
	\end{equation}
\end{lemma}

The existence and uniqueness of the solution of Problem \eqref{eq:disvar} follow from the Lax-Milgram theorem provided that $\gamma\ge\gamma_0$. The proof is standard and we refer to~\cite{Arnold:2001} for details. The explicit form of $C_{\text{inv}}$ may be found in~\cite{WarburtonHesthaven:2003}, which will be used to determine the lower bound $\gamma_0$ of the penalized parameter $\gamma$.

The following inequality slightly extends~\cite[Lemma 3.1]{HuangLuMing:2018}.
\begin{lemma}\label{lema:smalldomain}
For any $v\in H^s(D)$ with a positive number $s$, and for any subset $\Om\subset D$, there exists $C$ independent of~$\Om$ such that
	\begin{equation}\label{eq:smalldomain}
		\nm{v}{0,\Om}\le C\abs{\Om}^{\theta}\eta(\Om)\nm{v}{s,D},
	\end{equation}
	where $\theta=\min(s/n,1/2)$ and
	\[
	\eta(\Om){:}=\left\{\begin{aligned}
		\abs{\ln\abs{\Om}}^{1/2}\qquad&\text{if\;}n=2,3\text{\;and\;}s=n/2,\\		
		1\qquad&\text{otherwise}.
	\end{aligned}\right.
	\]
\end{lemma}

\begin{proof}
If $0<s<n/2$, then we proceed along the same line that leads to~\cite[Lemma 3.1]{HuangLuMing:2018} and obtain~\eqref{eq:smalldomain} with $\theta=s/n$.

For $s>n/2$, we use the embedding $H^s(D)\hookrightarrow L^\infty(D)$ and obtain
\[
	\nm{v}{0,\Om}\le\abs{\Om}^{1/2}\nm{v}{L^\infty(\Om)}\le\abs{\Om}^{1/2}\nm{v}{L^\infty(D)}\le C\abs{\Om}^{1/2}\nm{v}{s,D},
\]
where $C$ depends on $D$ but independent of $\Om$.
	
It remains to deal with $s=n/2$. The case $n=2,s=1$ has been proved in~\cite[Lemma 3.1]{HuangLuMing:2018}. The case $n=3,s=3/2$ may be proved as follows. Using~\cite[Theorem 2.1]{Kozono:2008}\footnote{The authors therein only considered $D=\R^n$, while the proof may be extended to the domain $D$ by Stein extension as in~\cite{AdamsFournier:2003}.}, for any $p>2$, there exists $C$ depending only on $D$ such that
\[
	\nm{v}{L^p(D)}\le C\sqrt{p}\nm{v}{3/2,D}.
\]
Therefore,
\begin{align*} \nm{v}{0,\Om}&\le\abs{\Om}^{1/2-1/p}\nm{v}{L^p(\Om)}\le\abs{\Om}^{1/2-1/p}\nm{v}{L^p(D)}\\
&\le C\abs{\Om}^{1/2-1/p}\sqrt{p}\nm{v}{3/2,D}.
\end{align*}

Taking $p=\abs{\ln\abs{\Om}}$ in the right-hand side of the above inequality, we obtain~\eqref{eq:smalldomain} for $n=3$ and $s=3/2$.
	
A combination of the above cases implies~\eqref{eq:smalldomain}.
\end{proof}

The following result concerns the accuracy of the method approximating the homogenized solution $u_0$ when $u_0$ is smooth.
\begin{theorem}\label{thm:Macro}
	Let $u_0$ and $v_h$ be the solutions of Problem~\eqref{eq:homoprob} and Problem~\eqref{eq:disvar}, respectively. If $u_0\in H^{1+s}(D)$ with $s>1/2$, and $\psi_g\in H_0^1(D)$ is the unique solution of the adjoint variational problem:
	\begin{equation}\label{eq:dual}
		\langle \A\na v,\na \psi_g\rangle=\langle g,v\rangle\qquad \text{for all\quad} v\in H_0^1(D),
	\end{equation}
	which satisfies the regularity estimate
	\begin{equation}\label{eq:reg}
		\nm{\psi_g}{2,D}\le C\nm{g}{0,D}.
	\end{equation}
	Then there exists $C$ depending on $D,\lam$ and $\Lam$ such that,
	\begin{equation}\label{eq:macroH1}
		\wnm{u_0-v_h}\le C\Lr{\inf_{\chi\in X_{h,H}^0}\wnm{ u_0-\chi}+\abs{K_1}^{\theta}\eta(K_1)+e(\textsc{HMM})},
	\end{equation}
	and
	\begin{equation}\label{eq:macroL2}
		\begin{aligned}
			\nm{u_0-v_h}{0,D}&\le C\Lr{\inf_{\chi\in X_{h,H}^0}\wnm{ u_0-\chi}+\abs{K_1}^{\theta}\eta(K_1)}\Lr{h+H+\abs{K_1}^{1/n}\eta(K_1)}\\
			&\quad+Ce(\textsc{HMM}).
		\end{aligned}
	\end{equation}
	
\end{theorem}%

Let $\chi$ be the Scott-Zhang interpolant~\cite{ScottZhang:1990} in~\eqref{eq:macroH1} and~\eqref{eq:macroL2} and we may obtain the following error estimates.
\begin{corollary}\label{coro:macro}
		Under the same assumptions as in Theorem~\ref{thm:Macro}, there exists $C$ depending on $D,\lam,\Lam$ and $\sigma$ such that
		\begin{equation*}
			\wnm{u_0-v_h}\le C\Lr{h^{s^*}+H^{s^*}+\abs{K_1}^{\theta}\eta(K_1)+e(\textsc{HMM})},
		\end{equation*}
		and
		\begin{equation*}
			\begin{aligned}
				\nm{u_0-v_h}{0,D}&\le C\Lr{h^{s^*}+H^{s^*}+\abs{K_1}^{\theta}\eta(K_1)}\Lr{h+H+\abs{K_1}^{1/n}\eta(K_1)}\\
				&\quad+Ce(\textsc{HMM}),
			\end{aligned}
		\end{equation*}
		with $s^*=\min(s,r)$.
	\end{corollary}

\noindent{\em Proof of Theorem \ref{thm:Macro}\;}	
Firstly, we need an auxiliary problem: Find $\wt{u}\in X_{h,H}^0$ such that
\begin{equation}\label{eq:hodisvar}
	A(\wt{u},v)=\langle f,v\rangle\quad \text{for all~} v\in X_{h,H}^0,
\end{equation}
where the bilinear form $A$ is defined for any $v,w\in X_{h,H}^0$ the same as $B_h^\eps$ with $\b_h$ replaced by $\A$.
From the Galerkin orthogonality
\begin{equation}\label{eq:Galer}
	A(u_0-\wt{u},v)=0 \qquad \text{for any}\quad v\in X_{h,H}^0,
\end{equation}
and~\eqref{eq:cont},~\eqref{eq:coer} we have
\begin{equation}\label{eq:cea}
	\wnm{u_0-\wt{u}}\leq \dfrac{4\max(\Lam,1)}{\min(1,\lam)}\inf_{\chi\in X_{h,H}^0}\wnm{u_0-\chi},
\end{equation}
and
\begin{equation}\label{eq:L2part1}
	\nm{u_0-\wt{u}}{0,D}\le 2\Lam \wnm{u_0-\wt{u}}\sup_{g\in L^2(D)}\dfrac{1}{\nm{g}{0,D}}
	\inf_{\chi\in X_{h,H}}\wnm{ \psi_g-\chi}.
\end{equation}

Let $w=\wt{u}-v_h$, we write
\begin{align*}
	B^\eps_h(w,w)=&B^\eps_h(\wt{u},w)-\dual{f}{w}=B^\eps_h(\wt{u},w)-A(\wt{u},w)\\
	=&\sum_{i=1}^2\int_{K_i}(\b_h-\A)\na \wt{u}\na w\dx-\sum_{e\in\E_\cap}\int_e\aver{(\b_h-\A)\na \wt{u}\cdot \n}_\omega\jump{w}\ds\\
	&-\sum_{e\in\E_\cap}\int_e\aver{(\b_h-\A)^T\na w\cdot \n}_\omega\jump{\wt{u}}\ds.
\end{align*}

By $\b_h-\A=\rho(\a-\A)+(1-\rho)(\A_h-\A)$, we have, for any $x\in K_1$,
\[
\nm{(\b_h-\A)(x)}{F}\le\nm{\b_h(x)}{F}+\nm{\A(x)}{F}\le\Lam'+\Lam,
\]
and for any $x\in K_2$,
\[
\nm{(\b_h-\A)(x)}{F}=\nm{(\A_h-\A)(x)}{F}\le e(\textsc{HMM}).
\]
This implies
\[
\abs{\sum_{i=1}^2\int_{K_i}(\b_h-\A)\na \wt{u}\na w\dx}\le(\Lam+\Lam')\abs{\wt{u}}_{1,K_1}\abs{w}_{1,K_1}+e(\textsc{HMM})\wnm{\wt{u}}\wnm{w}.
\]
Noting that $\rho=0$ for $e\in\mc{E}_{\cap}$, we obtain for any $x\in e$,
\[
\nm{(\b_h-\A)(x)}{F}=\nm{(\A_h-\A)(x)}{F}\le e(\textsc{HMM}).
\]
Hence, we obtain
\begin{align*}
	&\quad\abs{\sum_{e\in\E_\cap}\int_e\aver{(\b_h-\A)\na \wt{u}\cdot \n}_\omega\jump{w}\ds}\\
	&\le\Lr{\sum_{e\in\E_\cap}\dfrac{H_e+h_e}{\gamma}\nm{\aver{(\b_h-\A)\na \wt{u}\cdot \n}_\omega}{0,e}^2}^{1/2}\Lr{\sum_{e\in\E_\cap}\dfrac{\gamma}{H_e+h_e}\nm{\jump{w}}{0,e}^2}^{1/2}\\
	&\le\dfrac{e(\textsc{HMM})}{\gamma_0}\Lr{\sum_{e\in\E_\cap}(H_e+h_e)\nm{\aver{\na\wt{u}}_\omega}{0,e}^2}^{1/2}
	\wnm{w}.
\end{align*}
Using Cauchy-Schwarz inequality and ~\eqref{eq:traceinv}, we obtain
\begin{align*} \sum_{e\in\E_\cap}(H_e+h_e)\nm{\aver{\na\wt{u}}_\omega}{0,e}^2
&\le\sum_{e\in\E_\cap}(H_e+h_e)\Lr{\omega_1
	\nm{\na\wt{u}_1}{0,e}^2+\omega_2\nm{\na\wt{u}_2}{0,e}^2}\\
&=\sum_{e\in\E_h}h_e
	\nm{\na\wt{u}_1}{0,e}+\sum_{e\in\E_H}H_e\nm{\na\wt{u}_2}{0,e}^2\\
	&\le C_{\text{inv}}\Lr{\sum_{\tau\in\T_{h}^\Gamma}\nm{\na\wt{u}_1}{0,\tau}^2
+\sum_{\tau\in\T_{H}^\Gamma}\nm{\na\wt{u}_2}{0,\tau}^2}.
\end{align*}%

A combination of the above two inequalities gives
\[
\abs{\sum_{e\in\E_\cap}\int_e\aver{(\b_h-\A)\na \wt{u}\cdot \n}_\omega\jump{w}\ds}\le\dfrac{C_{\text{inv}}}{\gamma_0}e(\textsc{HMM})\wnm{\wt{u}}\wnm{w}.
\]
Exchanging the role of $\wt{u}$ and $w$, we obtain
\[
\abs{\sum_{e\in\E_\cap}\int_e\aver{(\b_h-\A)^T\na w\cdot \n}_\omega\jump{\wt{u}}\ds}\le
\dfrac{C_{\text{inv}}}{\gamma_0}e(\textsc{HMM})\wnm{\wt{u}}\wnm{w}.
\]

Combining all the above estimates, we obtain
\begin{equation}\label{eq:H1est1}
	B^\eps_h(w,w)\le(\Lam+\Lam')\abs{\wt{u}}_{1,K_1}\abs{w}_{1,K_1}+(1+2C_{\text{inv}}/\gamma_0)e(\textsc{HMM})\wnm{\wt{u}}\wnm{w}.
\end{equation}
Using the triangle inequality and~\eqref{eq:smalldomain},  we obtain
\begin{equation}\label{eq:H1est2}
	\begin{aligned}
		\nm{\wt{u}}{1,K_1}&\le\nm{u_0-\wt{u}}{1,K_1}+\nm{u_0}{1,K_1}\\
		&\le \wnm{u_0-\wt{u}}+C\abs{K_1}^{\theta}\eta(K_1)\nm{u_0}{1+s,D}.
	\end{aligned}	
\end{equation}

Combining the above three inequalities and the a-priori estimate $\wnm{\wt{u}}\leq C\nm{f}{-1,D}$, we obtain
\[
	\wnm{\wt{u}-v_h}\le C\Lr{\wnm{ u_0-\wt{u}}+\abs{K_1}^{\theta}\eta(K_1)\nm{u_0}{1+s,D}+e(\textsc{HMM})\nm{f}{-1,D}},
\]
which together with~\eqref{eq:cea} and the triangle inequality
conclude the estimate~\eqref{eq:macroH1}.

We exploit \textit{Aubin-Nitsche's dual argument} to prove the $L^2$ error bound.

For any $g\in L^2(D)$, let $\varphi_g\in H_0^1(D)$ satisfying
\[
B_h^\eps(v,\varphi_g)=\dual{g}{v}\qquad \text{for all\quad} v\in X_{h,H}.
\]
Substituting $v=\wt{u}-v_h$ into the above equation, we obtain
\begin{align*}
	\dual{g}{\wt{u}-v_h}&=B_h^\eps(\wt{u}-v_h,\varphi_g)=B_h^\eps(\wt{u},\varphi_g)-\dual{f}{\varphi_g}\\
	&=B_h^\eps(\wt{u},\varphi_g)-A(\wt{u},\varphi_g).
\end{align*}
Proceeding along the same line that leads to~\eqref{eq:H1est1}, we obtain
\[
\abs{\dual{g}{\wt{u}-v_h}}\leq C\Lr{\abs{\wt{u}}_{1,K_1}\abs{\varphi_g}_{1,K_1}+e(\textsc{HMM})\wnm{\wt{u}}\wnm{\varphi_g}}.
\]
Using~\eqref{eq:H1est2} and the a-priori estimates for $\wt{u}$ and $\varphi_g$, we obtain
\begin{align}\label{eq:dualest}
	\abs{\dual{g}{\wt{u}-v_h}}&\le C\abs{K_1}^{1/n}\eta(K_1)\nm{g}{0,D}\Lr{\wnm{u_0-\wt{u}}
		+\abs{K_1}^{\theta}\eta(K_1)\nm{u_0}{1+s,D}}\nn\\
	&\quad+Ce(\textsc{HMM})\nm{f}{-1,D}\nm{g}{-1,D}.
\end{align}
Then
\begin{align}\label{eq:L2part2}
	\nm{\wt{u}-v_h}{0,D}&\le C\abs{K_1}^{1/n}\eta(K_1)\Lr{\wnm{u_0-\wt{u}}+\abs{K_1}^{\theta}\eta(K_1)\nm{u_0}{1+s,D}}\\
	&\quad+Ce(\textsc{HMM})\nm{f}{-1,D}.\nn	
\end{align}
Combining the above inequality with~\eqref{eq:cea},~\eqref{eq:L2part1} and using the triangle inequality, we obtain~\eqref{eq:macroL2}.
\qed%

In the following, we shall consider the case when $u_0$ is nonsmooth, which may be caused by a rough homogenized coefficient $\A$, or a point load function $f$, or a nonconvex domain $D$. If $u_0\in H^{1+s}(D)$ with $0<s\le 1/2$, the Galerkin orthogonality~\eqref{eq:Galer} is invalid and we cannot use C\'{e}a's lemma to prove~\eqref{eq:cea}. To overcome this difficulty, we employ the medius analysis in~\cite{Gudi:2010_New}. We firstly need the following extra assumptions on $\T_{h,H}$:

\textbf{Assumption A}: $\T_h^\Gamma$ is quasi-uniform, i.e., there exists a constant $\nu$ independent of $h_\Gamma$ such that for any $\tau\in\T_h^{\Gamma}$, $h_\tau\geq\nu h_\Gamma$, where $h_\Gamma{:}=\max\{h_\tau :  \tau\in\T_h^\Gamma\}$.


\textbf{Assumption B}: $\E_h$ is a subgrid of $\E_H$, i.e., $\E_\cap=\E_h$.

According  to \textbf{Assumption A} and \textbf{Assumption B}, we may construct a compatible sub-decomposition $\wt{\T}_H^\Gamma$ out of $\T_H^\Gamma$ such that $\wt{\T}_H^\Gamma\cup\T_h^\Gamma$ is a conforming mesh, which is quasi-uniform near the interface $\Gamma$, while we never need such refined mesh in the implementation; see Fig.~\ref{fig:splitTH}. Both assumptions have been used in~\cite{Karakshian:2003} to prove the a-posterior error estimate for the DG method.
\begin{figure}[h]
	\centering
	\includegraphics{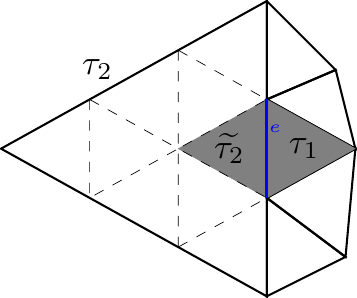}
	\caption{The split of the element in $\T_H^\Gamma$.}
	\label{fig:splitTH}
\end{figure}

We define the oscillation for $f\in L^2(D)$ as
\[
\textsc{osc}(f){:}=\left[\sum_{\tau\in \mathcal{T}_{h,H}}h_\tau^2\left(\inf_{\bar{f}\in P_0(\tau)}\nm{f-\bar{f}}{0,\tau}^2\right)\right]^{1/2},
\]
and the oscillation for $\A\in [L^\infty(D)]^{n\x n}$ as
\[
\textsc{osc}(\A):=\max_{\tau\in \T_{h,H}}\Lr{\inf_{\bar{\A}\in[P_0(\tau)]^{n\times n}}\nm{\A-\bar{\A}}{0,\infty,\tau}}.	
\]

By Meyers' regularity result~\cite{Meyers:1963}, there exists $p_0>2$ that depends on $D$, $\Lam$ and $\lam$, such that for all $p\leq p_0$,
\[
\abs{ u_0}_{1,p,D}\leq C\nm{f}{-1,p,D},
\]
where $C$ depends on $n$, $D$, $\Lam$ and $\lam$. Using H\"{o}lder's inequality, we obtain
\begin{equation}\label{eq:meyer}
	\begin{aligned}
		\abs{u_0}_{1,K_1}&\le\abs{K_1}^{1/2-1/p}\abs{u_0}_{1,p,K_1}\le\abs{K_1}^{1/2-1/p}\abs{u_0}_{1,p,D}\\
		&\leq C\abs{K_1}^{1/2-1/p} \nm{f}{-1,p,D}.
	\end{aligned}
\end{equation}

\begin{theorem}\label{thm:Macro2}
	Let $u_0$ and $v_h$ be the solutions of Problem~\eqref{eq:homoprob} and Problem~\eqref{eq:disvar}, respectively. If {\bf Assumption A} and {\bf Assumption B} are valid, then there exists $C$ depending on $D,\lam,\Lam$ and $C_\text{inv}$ such that for any $2<p<p_0$,
	\begin{equation}\label{eq:macroh1}
		\begin{aligned}
			\wnm{u_0-v_h}\leq& C\varrho\left(\inf_{\chi\in X_{h,H}^0}\wnm{ u_0-\chi}+\textsc{osc}(\A)+\textsc{osc}(f)\right)\\
			&+C\Lr{e(\textsc{HMM})+\abs{K_1}^{1/2-1/p}}.
		\end{aligned}
	\end{equation}
	
	Under the same assumptions, for any $2<p<\wt{p}_0$ with
	\begin{equation}\label{eq:index}
		\wt{p}_0=\left\{\begin{aligned}
			p_0&\quad\text{if\quad}n=2,\\
			\min(p_0,6)&\quad\text{if\quad}n=3,
		\end{aligned}\right.
	\end{equation}
	it holds that
	\begin{align}
		&\quad\nm{u_0-v_h}{0,D}\nn\\
		&\le C\Lr{\varrho^2\left(\inf_{\chi\in X_{h,H}^0}\wnm{ u_0-\chi}+\textsc{osc}(f)+\textsc{osc}(\A)\right)
			+\abs{K_1}^{1/2-1/p}}\nn\\
		&\quad\times \Bigg(\sup_{g\in L^2(D)}\frac{1}{\nm{g}{0,D}}\inf_{\wt{\chi}\in X_{h,H}}\Big(\sum_{\tau_\in \T_{h,H}}h_\tau^{-1}\nm{\psi_g-\wt{\chi}}{0,\tau}+\wnm{ \psi_g-\wt{\chi}}+\textsc{osc}(g)\Big)\nn\\
		&\qquad\quad+\textsc{osc}(\A)+\abs{K_1}^{1/2-1/p}\Bigg)+Ce(\textsc{HMM}),\label{eq:macrol2}
	\end{align}
	where $\varrho=(H+h)/(\gamma h_\Gamma)$ and $\psi_g\in H_0^1(D)$ is the unique solution of Problem~\eqref{eq:dual}.
\end{theorem}%

We do not impose the regularity assumption on $\psi_g$. The proof of this theorem is a combination of the ways that lead to Theorem~\ref{thm:Macro} and Lemma~\ref{lemma:Macro1low} below, provided that we replace~\eqref{eq:smalldomain} by ~\eqref{eq:meyer} and replace~\eqref{eq:dualest} by
\begin{align*}
	\abs{\dual{g}{\wt{u}-v_h}}\leq& C\abs{K_1}^{1/2-1/p}\Lr{\wnm{u_0-\wt{u}}
		+\abs{K_1}^{1/2-1/p}\nm{f}{-1,p,D}}\nm{g}{-1,p,D}\\
	&\quad+Ce(\textsc{HMM})\nm{f}{-1,D}\nm{g}{-1,D},\nn
\end{align*}
respectively.

It follows from the Sobolev imbedding $L^2(D)\hookrightarrow W^{-1,p}(D)$ for any $p\ge 1$ if $n=2$ and for any $1\leq p\le 6$ if $n=3$~\cite{AdamsFournier:2003}, the estimate~\eqref{eq:L2part2} is changed to
\begin{align}\label{eq:l2part2}
	\nm{\wt{u}-v_h}{0,D}\leq& C\abs{K_1}^{1/2-1/p}\Lr{\wnm{u_0-\wt{u}}+\abs{K_1}^{1/2-1/p}\nm{f}{-1,p,D}}\\
	&\quad+e(\textsc{HMM})\nm{f}{-1,D}\qquad\text{for any\quad}2<p<\wt{p}_0.\nn
\end{align}
\begin{lemma}\label{lemma:Macro1low}
	Under the same assumptions of Theorem~\ref{thm:Macro2}, there exists $C$ depending on $D,\lam,\Lam$ and $C_\text{inv}$ such that
	\begin{equation}\label{eq:h1part1}
		\wnm{u_0-\wt{u}}\leq C\varrho\left(\inf_{\chi\in X_{h,H}^0}\wnm{u_0-\chi}+\textsc{osc}(f)+\textsc{osc}(\A)\right),
	\end{equation}
	and
	\begin{align}
		\nm{u_0-\wt{u}}{0,D}&\leq C \Lr{\varrho\wnm{u_0-\wt{u}}+\textsc{osc}(f)+\textsc{osc}(\A)}\nn\\
		&\quad\times \Bigg(\sup_{g\in L^2(D)}\frac{1}{\nm{g}{0,D}}\Big(\inf_{\wt{\chi}\in X_{h,H}}\big(\sum_{\tau_\in \T_{h,H}}h_\tau^{-1}\nm{\psi_g-\wt{\chi}}{0,\tau}+\wnm{ \psi_g-\wt{\chi}}\big)\nn\\
		&\qquad\quad+\textsc{osc}(g)\Big)+\textsc{osc}(\A)\Bigg),\label{eq:l2part1}
	\end{align}
	where $\psi_g\in H_0^1(D)$ is the unique solution of Problem~\eqref{eq:dual}.
\end{lemma}%
\begin{remark}
By contrast to~\eqref{eq:cea} and~\eqref{eq:L2part1}, there is no regularity assumption on $u_0$ in Lemma~\ref{lemma:Macro1low}, while it contains extra oscillation terms $\textsc{osc}(\A)$, $\textsc{osc}(f)$ and $\textsc{osc}(g)$. Such terms are indispensable by the recent error estimates on Nitsche's methods~\cite{LuthenJuntunenStenberg:2018, GustafssonStenbergVideman:2019}. We also note that the right hand side of~\eqref{eq:h1part1} and~\eqref{eq:l2part1} depend on $\varrho$, which seems unpleasant at the first glance because the error blows up for large $\varrho$, while it may be small or even harmless if we tune the penalized parameter $\gamma$. This will be shown in Corollary~\ref{coro:macro2} and \S~\ref{sec:numerics}. It would be very interesting to know whether such dependance can be removed. We shall leave this for further study.
\end{remark}

The proof of the above lemma is quite involved and is of independent interest, we postpone it to~\S\ref{sec:pfaux}.

It follows from Theorem~\ref{thm:Macro2} that
\begin{corollary}\label{coro:macro2}
If $\A\in [\mc{C}^{0,\alpha}(D)]^{n\times n}$ with $\alpha\in (0,1)$ and $\beta=\min(\al,s)$, if
$\gamma>\max((H+h)/h_\Gamma,\gamma_0)$, $u_0\in H^{1+s}(D)$ with $0<s\le 1/2$ and $f\in L^2(D)$, then
\begin{equation}\label{eq:corro2h1}
			\wnm{u_0-v_h}\le C\Lr{h^{\beta}+H^{\beta}+\abs{K_1}^{s/n}\eta(K_1)+e(\textsc{HMM})}.
	\end{equation} 	

If the adjoint problem~\eqref{eq:dual} has the following regularity estimate
\[
\nm{\psi_g}{1+s,D}\leq \nm{g}{0,D},
\]
then
\begin{equation}\label{eq:corro2l2}
		\nm{u_0-v_h}{0,D}\le C\Lr{h^{2\beta}+H^{2\beta}+\abs{K_1}^{2s/n}\eta^2(K_1)+e(\textsc{HMM})}.
\end{equation} 	
\end{corollary}

\begin{proof}
If $\gamma>\max((H+h)/h_\Gamma,\gamma_0)$, then $\varrho$ is a constant  which is independent of any other parameters.
If $\A\in [\mc{C}^{0,\alpha}(D)]^{d\times d}$, then there exists $C$ independent of $h$ and $H$ such that \(\textsc{osc}(\A)\le C(h^{\alpha}+H^{\alpha})\). If $f\in L^2(D)$, then \(\textsc{osc}(f)\leq C(h+H)\nm{f}{0,D}\). Let $\chi$ be the Scott-Zhang interpolant of $u_0$ in~\eqref{eq:macroh1}. Replacing~\eqref{eq:meyer} by ~\eqref{eq:smalldomain} with $\theta=s/n$, we obtain  ~\eqref{eq:corro2h1}. 
	
For any given $g\in L^2(D)$, we have \(\textsc{osc}(g)\leq C(h+H)\nm{g}{0,D}\). Let $\chi$ and $\wt{\chi}$ be the Scott-Zhang interpolant of $u_0$ and $\phi_g$ in~\eqref{eq:macrol2}, respectively. In view of the regularity of $\psi_g$, we replace~\eqref{eq:l2part2} by
\begin{align*}
	\nm{\wt{u}-v_h}{0,D}&\le C\abs{K_1}^{s/n}\eta(K_1)\Lr{h^\beta+H^\beta+\abs{K_1}^{s/n}\eta(K_1)}\\
	&\quad+Ce(\textsc{HMM})\nm{f}{-1,D}.
\end{align*}
Combining all the estimates, we obtain~\eqref{eq:corro2l2}.
\end{proof}
\subsection{Accuracy for retrieving the local microscopic information}
We assume that $\text{dist}(K_0,\Gamma)=d\geq\kappa h$ for a sufficiently large $\kappa>0$. For a subset $B\subset D$, we define
\begin{align*}
	H^1_<(B)&{:}=\left\{u\in H^1(D)\ | \ u|_{D\setminus B}=0\right\},\quad\text{and}\\
	X_{<}(B)&{:}=\left\{u\in X_{h,H}\ | \ u|_{D\setminus B}=0\right\}.
\end{align*}

Let $G_1$ and $G$ be subsets of ${K_1}$ with $G_1\subset G$ and $\text{dist}(G_1,\pa G\setminus\pa D)=\wt{d} >0$. In order to prove the localized energy error estimate, we state  that some properties of the standard Lagrange finite element space $X_{h}$ hold~\cite{DemlowGuzmanSchatz:2011}:
\begin{enumerate}
	\item Local interpolant: There exists a local interpolant $Iu$ such that for any $u\in H^1_<(G_1)$, $Iu\in X_<(G)$.
	\item Inverse properties: For each $\chi\in X_<(K_1)$, and $\tau\in \mathcal{T}_h$, $1\leq p\leq q\leq \infty$, and $0\leq t\leq s\leq r+1$,
	\begin{equation}\label{eq:A2}
		\nm{\chi}{s,p,\tau}\leq C h_\tau^{t-s+n/p-n/q}\nm{\chi}{t,p,\tau}.
	\end{equation}
	\item Superapproximation: Let $\omega\in C^\infty({K_1})\cap H^1_<(G_1)$ with $\abs{\omega}_{j,\infty,K_1}\leq C\wt{d}^{-j}$ for integers $0\leq j\leq r+1$ for each $\chi\in X_<(G)\cap X_{h}^0 $ and for each $\tau\in\mathcal{T}_h$ satisfying $h_\tau\leq d$,
	\begin{equation}\label{eq:A3}
		\nm{\omega^2\chi-I(\omega^2\chi)}{0,\tau}\leq C\left(\frac{h_\tau}{\wt{d}}\abs{\omega \chi}_{1,\tau}+\frac{h_\tau}{\wt{d}^2}\nm{\chi}{0,\tau}\right).
	\end{equation}
\end{enumerate}

\begin{theorem}\label{thm:local}
	Let $u^\eps$ and $v_h$ be the solutions of \eqref{eq:prob} and \eqref{eq:hodisvar} respectively. Let $K_0\subset {K_1}\subset D$ be given, and let $\text{dist}(K_0,\Gamma)=d$. Let the above properties hold  with $\wt{d}=d/16$, in addition, let $\max_{\tau\cap\Gamma\not=\emptyset}h_{\tau}/d\le 1/16$. Then
	\begin{equation}\label{eq:MicroH1}
		\abs{u^\eps-v_h}_{1,K_0}\le C\left(\inf_{\chi\in X_{h}^0}\abs{u^\eps-\chi}_{1,K_1}+d^{-1}\nm{u^\eps-v_h}{0,K_1}\right),
	\end{equation}
	where $C$ depends only on $\lam'$, $\Lam'$, $\lam$, $\Lam$, and $D$.
\end{theorem}

\begin{proof}
	Let a subset $\wt{K}\subset K_1$ satisfying $\text{dist}(\wt{K},\Gamma)=d_0>0$. We have $X_<(\wt{K})\subset H^1_<(K_1)$. Then for any $v,w\in H^1_<(K_1)$, the bilinear form $B^\eps_h(v,w)$ degenerates to  $\langle b^\eps_h v, w\rangle$. Thus, the proof of above theorem is the same with that in \cite[Theorem 3.2]{HuangLuMing:2018}, and we omit the details for brevity.
\end{proof}

Using the triangle inequality, we have
\[
\abs{u^\eps-v_h}_{1,K_0}\le C\Lr{\inf_{\chi\in X_{h,H}^0}\abs{u^\eps-\chi}_{1,K_1}+d^{-1}\Lr{\nm{u_0-v_h}{0,D}+\nm{u_0-u^\eps}{0,D}}}.
\]

The first term in the right-hand side of the above inequality concerns how the local events are resolved. Accurate approximation requires a highly refined mesh, which is allowed by Theorem~\ref{thm:local}; using Corollary~\ref{coro:macro} and Corollary~\ref{coro:macro2}, we may bound $d^{-1}\nm{u_0-v_h}{0,D}$; the last term converges to zero as $\eps$ tends to zero by H-convergence theory~\cite{Tartar:2009} for any bounded and measurable $\a$. Therefore, the method converges for Problem~\eqref{eq:prob} with bounded and measurable coefficient. Moreover, if we assume more structural conditions such as periodicity, almost-periodicity or stochastic ergodicity on $\a$, we may expect $\nm{u^\eps-u_0}{0,D}\simeq\mc{O}(\eps^{\alpha})$ for certain $\alpha>0$. We refer to~\cite{KenigLinShen:2012, Shen:2015, Armstrong:2019} for extensive discussions for such $L^2-$estimate.
\section{Proof of Lemma~\ref{lemma:Macro1low}}\label{sec:pfaux}

To prove Lemma~\ref{lemma:Macro1low}, we employ the medius analysis in~\cite{Gudi:2010_New}. We need an enriching operator $E_h: X_{h,H}\rightarrow X_{h,H}\cap H^1(D)$. The construction of $E_h$ is similar to that in~\cite{Karakshian:2003} with certain modifications.

Let $\mc{N}(\tau)$ be the set of all nodes of $\tau\in\T_{h,H}$. Define $\mc{N}_H{:}=\bigcup_{\tau\in\mc{T}_H}\mc{N}(\tau)$, $\mc{N}_h=\bigcup_{\tau\in\mc{T}_h}\mc{N}(\tau)$, $\mc{N}_\Gamma{:}=\sset{p\in\mc{N}_h\cup\mc{N}_H}{p\in\Gamma}$ and $\mc{N}_\Gamma^0{:}=\N_H\cap\N_\Gamma$.
\begin{lemma}\label{lemma:enrich}
	If {\bf Assumption A} and {\bf Assumption B} hold for  $n=2,3$, there exists a linear map $E_h:\ X_{h,H}\longrightarrow X_{h,H}\cap H^1(D)$ satisfying
	\begin{equation}\label{eq:approx1}
		\sum_{\tau\in T_{h,H}^\Gamma}h_\tau^{2m-2}\abs{v-E_hv}_{m,\tau}^2\leq C\sum_{e\in\E_\cap}h_e^{-1}\nm{\jump{v}}{0,e}^2\quad\text{for all}\  m\leq r,
	\end{equation}
	where the constant $C$ is independent of $h$, $H$ and $H/h$.
\end{lemma}
\begin{proof}
	See Appendix A.	
\end{proof}

A direct consequence of the above local enriching estimate is the following
\begin{corollary}\label{coro:enrich}
If \textbf{Assumption A} and \textbf{Assumption B} are true, then there exists $C$ independent of $h,H$, $\gamma$ and the mesh ratio $H/h$ such that
\begin{equation}\label{eq:approx0}
		\sum_{\tau\in T_{h,H}^\Gamma}h_\tau^{2m-2}\abs{v-E_hv}_{m,\tau}^2\leq C\varrho\wnm{v}^2,\quad  0\le m\le r,
\end{equation}
and	
\begin{equation}\label{eq:approx2}
		\wnm{v-E_hv}^2\le C\varrho\wnm{v}^2,
	\end{equation}
	and
	\begin{align}
		\sum_{e\in\E_\cap}h_e^{-1}\nm{\aver{v-E_hv}^\omega}{0,e}^2&\le C\varrho\wnm{v}^2,\label{eq:enrich-trace}\\
		\sum_{e\in\E_\cap}(H_e+h_e)\abs{\aver{v-E_hv}}_{1,e}^2&\le C\varrho\wnm{v}^2.\label{eq:enrich-trace1}
	\end{align}
\end{corollary}

\begin{proof}
	Using \textbf{Assumption A}, we have
	\[
	\sum_{e\in\E_\cap}h_e^{-1}\nm{\jump{v}}{0,e}^2\le \sum_{e\in\E_\cap}\dfrac{H_e+h_e}{\gamma h_e}\dfrac{\gamma}{h_e+H_e}\nm{\jump{v}}{0,e}^2\le\varrho\wnm{v}^2,
	\]
	Therefore, the estimate~\eqref{eq:approx0} is a direct consequence of~\eqref{eq:approx1}.
	
	Since $v-E_hv\equiv 0$ in $\T_{h,H}\setminus \T_{h,H}^\Gamma$, using the above inequality and~\eqref{eq:approx0} with $m=1$, we obtain~\eqref{eq:approx2}.
	
	For any $v=(v_1,v_2)\in X_h\times X_H$, by definition, we obtain
	\[
	\sum_{e\in\E_\cap}h_e^{-1}\nm{\aver{v-E_hv}^\omega}{0,e}^2
	\le\sum_{e\in\E_\cap}\Lr{h_e^{-1}\nm{v_{1}-E_hv}{0,e}^2+h_e^{-1}\omega_1^2\nm{v_2-E_hv}{0,e}^2}.
	\]
	Using \textbf{Assumption A}, we have, for any $e\in\E_\cap$,
	\[
	h_e^{-1}\omega_1^2\le h_e\omega_1\le H_e^{-1}.
	\]
	Using the trace inequality, there exists $C$ such that
	\begin{align*}
		\sum_{e\in\E_\cap}h_e^{-1}\nm{\aver{v-E_hv}^\omega}{0,e}^2
		&\le\sum_{e\in\E_h}h_e^{-1}\abs{v_1-E_hv}_{1,e}^2+\sum_{E\in\E_H}H_E^{-1}\abs{v_2-E_hv}_{1,E}^2\\
		&\le C\sum_{\tau\in\T_{h,H}^\Gamma}\Lr{h_\tau^{-2}\nm{v-E_hv}{0,\tau}^2+\abs{v-E_hv}_{1,\tau}^2},
	\end{align*}
	which together with~\eqref{eq:approx0} implies~\eqref{eq:enrich-trace}.
	
	Proceeding along the same line, we obtain
	\begin{align*}
		\sum_{e\in\E_\cap}(H_e+h_e)\abs{\aver{v-E_hv}}_{1,e}^2&\le\sum_{e\in\E_h}h_e\abs{v_1-E_hv}_{1,e}^2+\sum_{E\in\E_H}H_E\abs{v_2-E_hv}_{1,E}^2\\
		&\le C\sum_{\tau\in\T_{h,H}^\Gamma}\Lr{\abs{v-E_hv}_{1,\tau}+h_\tau^2\abs{v-E_hv}_{2,\tau}},
	\end{align*}
	which together with~\eqref{eq:approx0} implies~\eqref{eq:enrich-trace1}.
\end{proof}

Next lemma concerns the error estimate of $\wt{u}$ approximating $u_0$, which is a type of \textit{Strang's lemma}, the proof follows from the same line of~\cite[Lemma 2.1 and \S~{3.2}]{Gudi:2010_New} and we omit the details.
\begin{lemma}\label{lemma:Strang}	
	There exists a constant $C$ such that
	\[
	\wnm{u_0-\wt{u}}\leq C\inf_{\chi\in X_{h,H}^0}\left(\wnm{u_0-\chi}+\sup_{\psi\in X_{h,H}^0\setminus \{0\}}\dfrac{\dual{f}{\psi-E_h\psi}-A(\chi,\psi-E_h\psi)}{\wnm{\psi}}\right).
	\]
\end{lemma}%

To estimate the second term in the right hand side of the above inequality, we shall use the following lemma, which is similar to the discrete local efficiency estimates in the a-posteriori error analysis~\cite{Verfurth:1994}. The proof is quite standard and we refer to~\cite{Verfurth:2013}.
\begin{lemma}\label{lemma:post}
	Let $u_0\in H_0^1(D)$ be the solution of \eqref{eq:homoprob} and  $\chi\in X_{h,H}^0$.
	\begin{enumerate}
		\item For each $\tau\in\T_{h,H}$, let $\bar{\A}\in [P_0(\tau)]^{n\times n}$, there exists $C$ such that
		\begin{equation}\label{eq:post1}
			\begin{aligned}
				h_\tau^2\nm{f+\na\cdot\bar{\A}\na \chi}{0,\tau}^2\leq& C\Bigl(\abs{ u_0-\chi}_{1,\tau}^2+h_\tau^2\inf_{\bar{f}\in P_0(\tau)}\nm{f-\bar{f}}{0,\tau}^2\\
				&\phantom{\Bigl(}\quad+\nm{\A-\bar{\A}}{0,\infty,\tau}^2\abs{ u_0}_{1,\tau}^2\Bigr).
			\end{aligned}
		\end{equation}
		
		\item For each $e\in\E_\cap$, there exists $(\tau_1,\wt{\tau_2})\in \T_h^\Gamma\times \wt{\T}_H^\Gamma$ such that $\tau_1\cap \wt{\tau_2}=e$ and $e$ is a full edge of both $\tau_1$ and $\wt{\tau_2}$; See Fig.~\ref{fig:splitTH}. Let $\bar{\A}\in [P_0(\tau_1)\times P_0(\wt{\tau_2})]^{n\times n}$, we have
\begin{align}
				h_e\nm{\jump{\bar{\A}\na \chi\cdot \n}}{0,e}^2
				&\le C\sum_{\tau\in\{\tau_1,\wt{\tau_2}\}}\Bigl(\nm{u_0-\chi}{0,\tau}^2
				+h_\tau^2\inf_{\bar{f}\in P_0(\tau)}\nm{f-\bar{f}}{0,\tau}^2\nn\\
				&\phantom{C\sum_{\tau\in\{\tau_1,\wt{\tau_2}\}}
\Bigl(}\quad+\nm{\A-\bar{\A}}{0,\infty,\tau}^2\abs{u_0}_{1,\tau}^2\Bigr).\label{eq:post2}
\end{align}
\end{enumerate}
\end{lemma}


Combining all the above estimates, we are ready to prove Lemma~\ref{lemma:Macro1low}.

\noindent{\em Proof of Lemma~\ref{lemma:Macro1low}\;}
On each element $\tau\in\T_{h,H}$, let $\bar{\A}\in [P_0(\tau)]^{n\times n}$. We define a bilinear form $\bar{A}(\cdot,\cdot)$ the same as $A(\cdot,\cdot)$ with $\A$ replaced by $\bar{\A}$. For any $\psi\in X_{h,H}^0$, using~\eqref{eq:magic}, we obtain
	\begin{align*}
		&\dual{f}{\psi-E_h\psi}-A(\chi,\psi-E_h\psi)\\
		=&\dual{f}{\psi-E_h\psi}-\bar{A}(\chi,\psi-E_h\psi)+\bar{A}(\chi,\psi-E_h\psi)-A(\chi,\psi-E_h\psi)\\
		=&\sum_{\tau\in \mathcal{T}_{h,H}}\int_{\tau}(f+\na\cdot\bar{\A}\na \chi)(\psi-E_h\psi)\dx-\sum_{e\in\E_\cap}\int_e\jump{\bar{\A}\na \chi\cdot \n}\aver{\psi-E_h\psi}^\omega\ds\\
		&\quad+\sum_{e\in\E_\cap}\int_e\aver{\bar{\A}^T\na (\psi-E_h\psi)\cdot \n}_\omega\jump{u_0-\chi}\ds \\
		&\quad-\sum_{e\in\E_\cap}\int_e\frac{\gamma}{H_e+h_e}\jump{u_0-\chi}\jump{\psi-E_h\psi}\ds+\Lr{\bar{A}(\chi,\psi-E_h\psi)-A(\chi,\psi-E_h\psi)}\\
		=&I_1+\cdots+I_5.
	\end{align*}
	
	Using~\eqref{eq:post1} and~\eqref{eq:approx0} with $m=0$, we obtain
	\begin{align*}
		\abs{I_1}&\le
		\Lr{\sum_{\tau\in \mathcal{T}_{h,H}}h_{\tau}^2\nm{f+\na\cdot\bar{\A}\na \chi}{0,\tau}^2}^{1/2}
		\Lr{\sum_{\tau\in \mathcal{T}_{h,H}^\Gamma}h_{\tau}^{-2}\nm{\psi-E_h\psi}{0,\tau}^2}^{1/2}\\
		&\le C\varrho\Lr{\inf_{\chi\in X_{h,H}^0}\abs{u_0-\chi}_{1,D}+\textsc{osc}(f)+\textsc{osc}(\A)}\wnm{\psi}.
	\end{align*}
	
	Using~\eqref{eq:post2} and~\eqref{eq:enrich-trace}, we obtain
	\begin{align*}
		\abs{I_2}&\le
		\Lr{\sum_{e\in\E_\cap}h_e\nm{\jump{\bar{\A}\na \chi\cdot \n}}{0,e}^2}^{1/2}\Lr{\sum_{e\in\E_\cap}h_e^{-1}\nm{\aver{\psi-E_h\psi}^\omega}{0,e}^2}^{1/2}\\
		&\le
		C\varrho\Lr{\inf_{\chi\in X_{h,H}^0}\abs{u_0-\chi}_{1,D}+\textsc{osc}(f)+\textsc{osc}(\A)}\wnm{\psi}.
	\end{align*}
	
	Using~\eqref{eq:enrich-trace1}, we obtain
	\begin{align*}
		\abs{I_3}&\le C\Lr{\sum_{e\in\E_\cap}\dfrac{H_e+h_e}{\gamma}\abs{\aver{\psi-E_h\psi}}_{1,e}^2}^{1/2}\Lr{\sum_{e\in \E_\cap}\frac{\gamma}{H_e+h_e}\nm{\jump{u_0-\chi}}{0,e}^2}^{1/2}\\
		&\le C(\varrho/\gamma_0)\wnm{u_0-\chi}\wnm{\psi}.
	\end{align*}
	
	Using~\eqref{eq:approx2}, we may bound the last two terms as
	\[
	\abs{I_4}\le C\wnm{u_0-\chi}\wnm{\psi},
	\]
	and
	\begin{align*}
		\abs{I_5}\le C\textsc{osc}(\A)\wnm{\chi}\wnm{\psi-E_h\psi}
		\le C\varrho\,\textsc{osc}(\A)\Lr{\wnm{u_0-\chi}+\abs{u_0}_{1,D}}\wnm{\psi}.
	\end{align*}
	Combining all the above estimates and using Lemma~\ref{lemma:Strang}, we obtain~\eqref{eq:h1part1}.
	
Proceeding along the same line that leads to~\cite[Theorem 4.4]{Gudi:2010_New}, we obtain $L^2$ estimate~\eqref{eq:l2part1}, and we postpone the proof to \textbf{Appendix B}.
\qed
\section{Numerical Experiments}\label{sec:numerics}
In this part, we report two examples with different shapes of defects to demonstrate the accuracy and efficiency of the method. The governing equation is~\eqref{eq:prob} with domain $D=(0,1)^2$,$f=1$ and the homogeneous Dirichlet boundary condition is imposed on $\pa D$. The finite element solvers are carried on FreeFem++ toolbox~\cite{Hecht:2012}\footnote{https://freefem.org/}.

The first example is  taken from~\cite{HuangLuMing:2018} with a two-scale coefficient.
\textbf{Example 1}:\label{Ex1}
Let
\begin{equation}\label{eq:ex1}
	a^\eps(x)=\frac{(R_1+R_2\sin(2\pi x_1)(R_1+R_2\cos(2\pi x_2))}{(R_1+R_2\sin(2\pi x_1/\eps))(R_1+R_2\sin(2\pi x_2/\eps))}I,
\end{equation}
where $I$ is a two by two identity matrix. The effective matrix is given by
\begin{equation}\label{eq:ex1homo}
	\A(x)=\frac{(R_1+R_2\sin(2\pi x_1)(R_1+R_2\cos(2\pi x_2))}{R_1\sqrt{R_1^2-R_2^2}}I.
\end{equation}
In the simulation we let $R_1=2.5$, $R_2=1.5$ and $\eps=0.01$. The reference
solutions for $u^\eps$ and $u_0$ are obtained by solving Problem~\eqref{eq:prob} and~\eqref{eq:homoprob} with
linear element over a uniform mesh with the mesh size around $3.33e-4$.

The second example is taken from~\cite{AbdulleJecker:2015} and the coefficient has no clear scale separation inside $K_0$.
\textbf{Example 2}:\label{Ex2}
For a subset $B\subset D$, we define $\delta_B$ as the characteristic function for the set $B$. The setup for this example is the same with the first one except that the coefficient is replaced by $a^\eps=\delta_{K_0}\tilde{a}+(1-\delta_{K_0})\tilde{a}^\eps$, where
\[
\tilde{a}(x)=\left(3+\frac{1}{7}\sum_{j=0}^{4}\sum_{i=1}^{j}\frac{1}{j+1}\cos\left(\left\lfloor8\left(ix_2-\frac{x_1}{i+1}\right)\right\rfloor+\lfloor150ix_1\rfloor+\lfloor150x_2\rfloor\right)\right)I,
\]
and
\[
\tilde{a}^\eps(x)=(2.1+\cos(2\pi x_1/\eps)\cos(2\pi x_2/\eps)+\sin(4x_1^2x_2^2))I.
\]
We let $\eps=0.0063$ in the simulation.

By~\cite{HuangLuMing:2018},  the effective matrix $\A=\delta_{K_0}\tilde{a}+(1-\delta_{K_0})\tilde{A}$, where $\tilde{A}$ is the effective matrix associated with  $\tilde{a}^\eps$. Since there is no analytical formula for $\tilde{A}$, we use the least-squares reconstruction method in~\cite{LiMingTang:2012, HuangMingSong:2020} to obtain a higher-order approximation to $\wt{A}$ so that $e(\text{HMM})$ is negligible. The approximation to $\wt{A}$ is denoted by $\tilde{A}_h$. The homogenized solution $u_0$ is computed by solving Problem~\eqref{eq:homoprob} with $\A_h=\delta_{K_0}\tilde{a}+(1-\delta_{K_0})\tilde{A}_h$. The hybrid coefficient is
\[
b^\eps_h=\delta_{K_0}\tilde{a}+(1-\delta_{K_0})(\rho\tilde{a}^\eps+(1-\rho)\tilde{A}_h).
\]
We use linear element  to compute $u^\eps$ and $u_0$ over a refined $3000\times 3000$ mesh as the reference solution.

For both examples, our major interests are the following two quantities:
\begin{equation}\label{eq:testerr}
	e(u_0){:}=\frac{\abs{ u_0-v_h}_{1,K_2}}{\abs{ u_0}_{1,K_2}}\quad \text{and}\quad
	e(u^\eps){:}=\frac{\abs{u^\eps-v_h}_{1,K_0}}{\abs{ u^\eps}_{1,K_0}}.
\end{equation}
%
\subsection{The choice of the penalized factor $\gamma$}
Based on the explicit expressions of the inverse trace inequalities in~\cite{WarburtonHesthaven:2003}, the lower bound $\gamma_0$ in Lemma~\ref{lemma:coer} for the penalized factor may be bounded by
\begin{equation}\label{eq:gamma}
	\gamma_0\leq\left\{\begin{aligned}
		\dfrac{16\sqrt{3}}{9}\dfrac{r(r+1)\sigma\Lam'^2}{\min\{1,\lam'^2\}}\qquad&n=2,\\
		\dfrac{32\sqrt{3}}{27}\dfrac{r(r+2)\sigma\Lam'^2}{\min\{1,\lam'^2\}}\qquad&n=3.
	\end{aligned}\right.
\end{equation}

We test \textbf{Example 1} and \textbf{Example 2} with the well defect (see \S~\ref{sect:well}) for different $\gamma$ and the results are plotted in Fig.~\ref{fig:gamma}.
\begin{figure}[tbhp]
	\centering
	{\includegraphics[width=5.8cm]{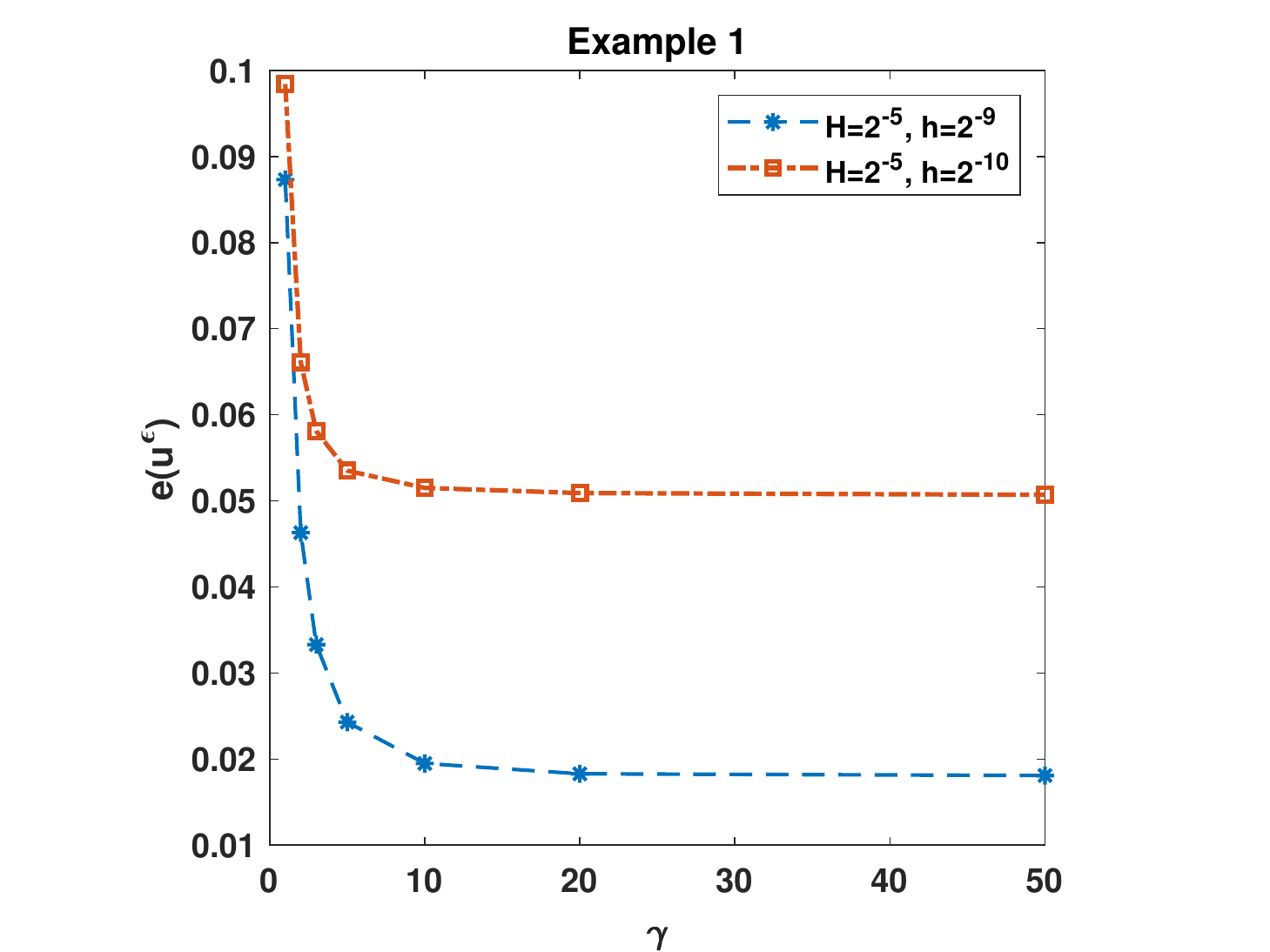}}
	{\includegraphics[width=5.8cm]{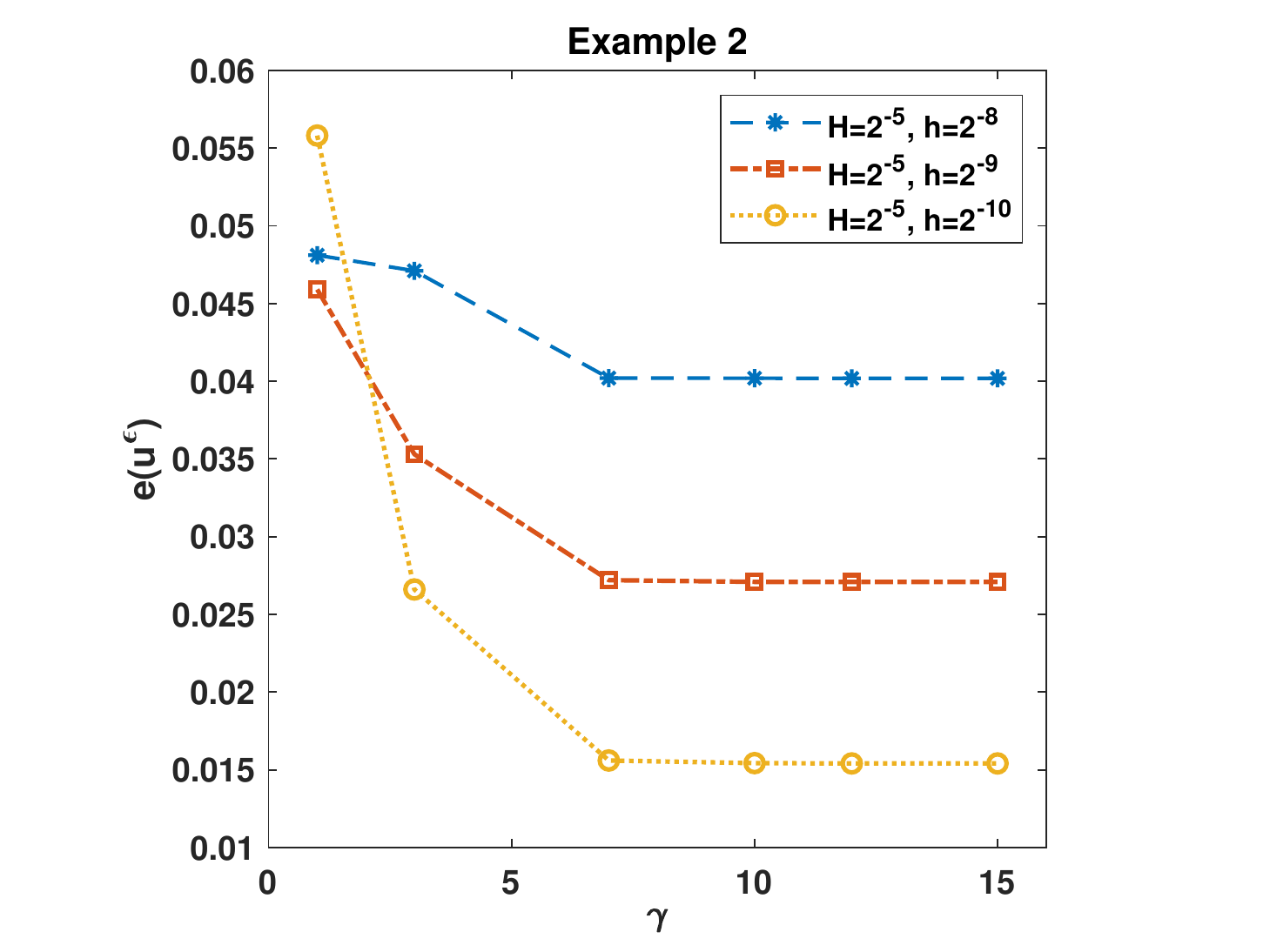}}
	\caption{$\abs{u^\eps-v_h}_{1,K_0}/\abs{ u^\eps}_{1,K_0}$ versus $\gamma$ in the well defect.}
	\label{fig:gamma}
\end{figure}
We observe that the error does not change when $\gamma$ is bigger than certain threshold value.  It is worthwhile to mention that the threshold value is independent of $H/h$. In what follows, we set $\gamma=50$ in \textbf{Example 1} and $\gamma=20$ in \textbf{Example 2}.
\subsection{The effect of the smoothness of the transition function}\label{sect:rho}
%

To study the effect of the smoothness of $\rho$ numerically, we consider a square defect $K_0{:}=(0.5,0.5)+(-L,L)^2$ with $L=0.05$, and define $K_1{:}=(0.5,0.5)+(-L-\delta,L+\delta)^2$ with $\delta=0.05$. The transition function $\rho=\mu(x_1-0.5)\mu(x_2-0.5)$  with $\mu:\mathbb{R}\rightarrow[0,1]$ defined by
\[
\mu(t){:}=\left\{\begin{aligned}
	\dfrac12\cos\Lr{\pi(x+L)/\delta}+\dfrac12, \qquad&\text{if\quad} t\in[-L-\delta,-L),\\
	1\qquad&\text{if\quad}t\in(-L,L),\\
	\dfrac12\cos\Lr{\pi(-x+L)/\delta}+\dfrac12\qquad&\text{if\quad} t\in(L,L+\delta],\\
	0\qquad &\text{otherwise}.
\end{aligned}\right.
\]
Such $\rho$ is a C$^1$ function. We test \textbf{Example 1} with this C$^1$ transition function and the one constructed by a linear interpolant introduced in \S~\ref{sec:method}, which is a continuous function. The results are reported in Table~\ref{table:rho}. We observe that there is no significant difference between these two choices and we shall use the $C^0$ transition function in the following tests.
\begin{table}[tbhp]
	{\footnotesize
		\caption{The error of Example 1 with different $\rho$.}\label{table:rho}
		\begin{center}
			\begin{tabular}{|c|cc|}
				\hline
				$H=2^{-5},h=2^{-8}$&$e(u_0)$
				&$e(u^\eps)$\\
				\hline
				$C^1$ transition func. &4.21e-2&1.12e-1\\
				$C^0$ transition func. &4.21e-2&1.13e-1\\
				\hline
				$H=2^{-6},h=2^{-9}$&$e(u_0)$&$e(u^\eps)$\\
				\hline
				$C^1$ transition func. &1.62e-2&4.62e-2\\	
				$C^0$ transition func. &1.62e-2&4.65e-2\\
				\hline
			\end{tabular}
	\end{center}}
\end{table}
\subsection{The accuracy of the Nitsche hybrid method}\label{sect:cases}
We test three defects with different shapes: the well defect, the channel defect
and the ellipse defect.
%
\subsubsection{Well defect}\label{sect:well}
The defect $K_0=(0.5,0.5)+(-L,L)^2$ with $L=0.05$ is a square, and we define $K_1{:}=(0.5,0.5)+(-L-\delta,L+\delta)^2$ with $\delta=0.05$. We report the results for \textbf{Example 1} in Table~\ref{table:fixH_well1} by fixing $H=2^{-5}$ and decreasing $h$. We observe that  locally refined mesh resolves the local events, and $e(u^\eps)\simeq \mathcal{O}(h)$,which is
consistent with the following explicit form of the error estimate~\eqref{eq:MicroH1}:
\[
\abs{u^\eps-v_h}_{1,K_0}\le C\Lr{h+\dfrac{1}{\delta}\Lr{H^2+L^2\abs{\ln L}+\eps^{s-1}}}.
\]
The last term comes from the following error estimate proved in~\cite{Moskow:1997}:
\[
\nm{u^\eps-u_0}{0,D}\le C\eps^{s-1}\nm{u_0}{s,D} \quad \text{for~} s<3.
\]
The error $e(u_0)$ remains unchanged when $h$ is decreased, which shows that the resolution inside the defect has negligible effect on the accuracy for retrieving the macroscopic information.
\begin{table}[tbhp]
	{\footnotesize
		\caption{Errors of Example 1 for the well defect with a fixed mesh $H=2^{-5}$.}
		\label{table:fixH_well1}
		\begin{center}
			\begin{tabular}{|c|ccc|}
				\hline
				h&$2^{-7}$&$2^{-8}$&$2^{-9}$\\
				\hline
				$e(u^\eps)$& 2.32e-1& 1.15e-1& 4.60e-2\\
				rate&&1.01&1.32\\
				\hline
				$e(u_0)$& 4.60e-2& 4.67e-2& 4.61e-2\\
				\hline
			\end{tabular}
	\end{center}}
\end{table}%

To compute $e(u_0)$,
we fix the mesh size in $K_1$ and refine the mesh in $K_2$. The result is reported in Table~\ref{table:fixh_well1}. When $H\simeq L$, the first order rate of convergence is observed for the error $e(u_0)$, which is consistent with
\[
\abs{u_0-v_h}_{1,K_2}\leq C(h+H+L\abs{\ln L}^{1/2}).
\]
Nevertheless, the error $e(u^\eps)$ remains unchanged when $H$ is decreased.
\begin{table}[tbhp]
	{\footnotesize
		\caption{The error of Example 1 in the well defect with a fixed $h=2^{-10}$.}
		\label{table:fixh_well1}
		\begin{center}
			\begin{tabular}{|c|ccc|}
				\hline
				H&$2^{-4}$&$2^{-5}$&$2^{-6}$\\
				\hline
				
				$e(u_0)$& 1.18e-1& 4.61e-2& 1.76e-2\\
				rate&&1.35&1.39\\
				\hline
				$e(u^\eps)$& 1.96e-2& 1.69e-2& 1.66e-2\\
				\hline
			\end{tabular}
	\end{center}}
\end{table}

We turn to \textbf{Example 2}. It follows from Table~\ref{table:fixH_well2} that the rate of convergence
for $e(u_0)$ is bigger than $1$, while we do not know any quantity estimate on
$e(\hmm)$ in this case.
Fig.~\ref{fig:well_noscale} indicates that the   error $e(u^\eps)$converges at a rate around $0.71$, which deteriorates a little bit than \textbf{Example 1}. This may be due to the roughness of $a^\eps$ inside $K_0$.
\begin{table}[tbhp]
	{\footnotesize
		\caption{Errors of Example 2 for the well defect with a fixed $h=2^{-10}$.}
		\label{table:fixH_well2}
		\begin{center}
			\begin{tabular}{|c|ccc|}
				\hline
				$H$&$2^{-4}$&$2^{-5}$&$2^{-6}$\\
				\hline
				$e(u_0)$& 1.06e-1& 4.90e-2& 1.56e-2\\
				rate&&1.11&1.65\\
				\hline
				$e(u^\eps)$& 1.56e-2& 1.02e-2& 1.02e-2\\
				\hline
			\end{tabular}
	\end{center}}
\end{table}%

\begin{figure}[tbhp]
	\centering
	{\includegraphics[width=6cm]{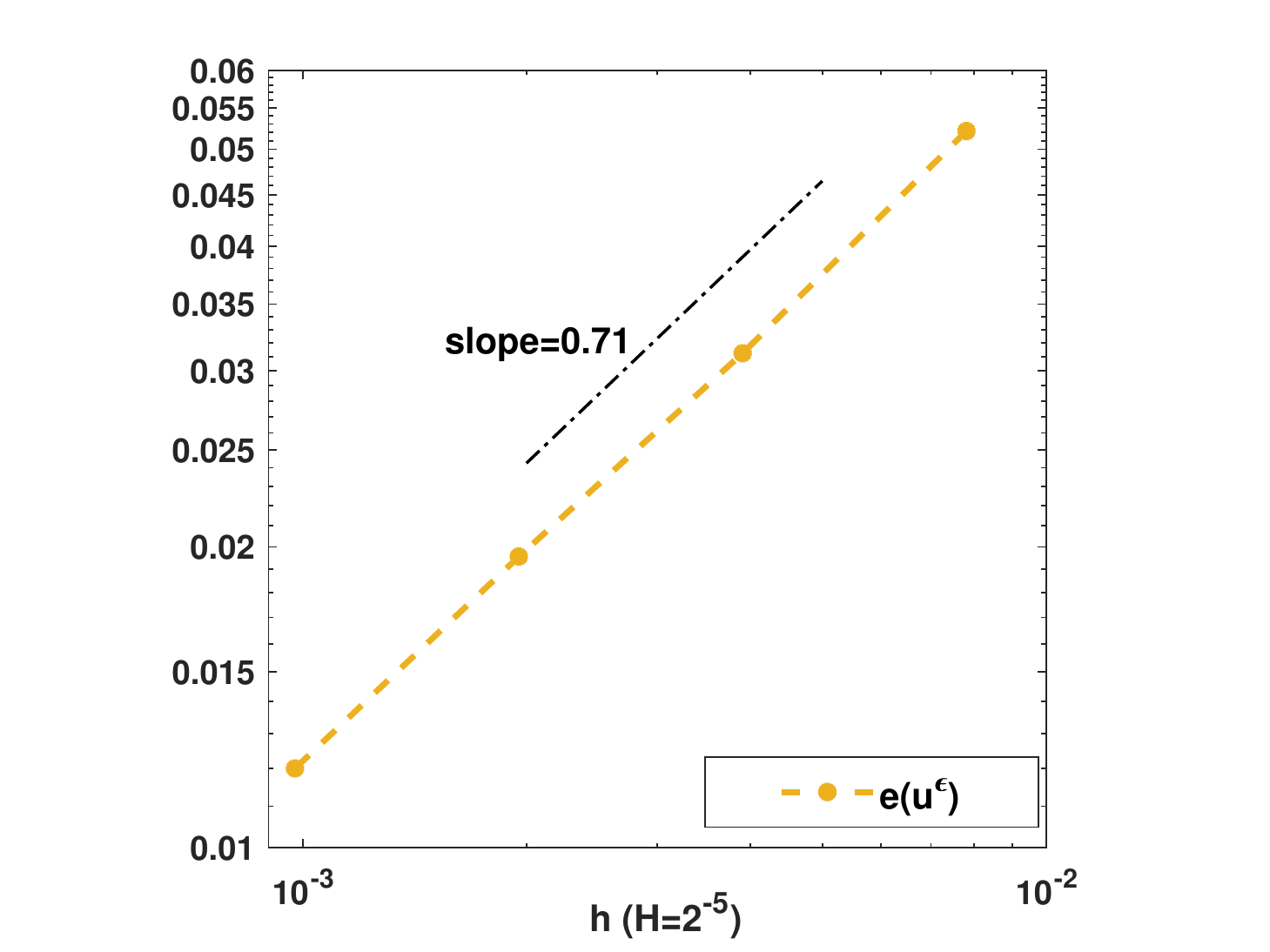}}
	\caption{Errors of \textbf{Example 2} for the well defect with a fixed $H=2^{-5}$.}
	\label{fig:well_noscale}
\end{figure}
\subsubsection{Channel defect}
Let $K_0$ be a channel with a corner. with width  $L=0.05$ and $K_1$ is the set within a distance of $\delta=0.025$ away from the channel; see Fig.~\ref{fig:rho}$_b$.
%

We firstly test \textbf{Example 1} with different $H$. The result is shown in Table~\ref{table:fixh_cha1}. We observe that the first order rate of convergence for the error $e(u_0)$, which is consistent with the theoretical result.
%
\begin{table}[tbhp]
	{\footnotesize
		\caption{The error of Example 1 in the channel defect with $h=2^{-9}$.}
		\label{table:fixh_cha1}
		\begin{center}
			\begin{tabular}{|c|ccc|}
				\hline
				H&$2^{-4}$&$2^{-5}$&$2^{-6}$\\
				\hline
				$e(u_0)$& 9.74E-1& 4.86E-2& 2.40E-2\\
				rate&&1.19&1.00\\
				\hline
				$e(u^\eps)$& 7.74e-2& 7.73e-2& 7.73e-2\\
				\hline
			\end{tabular}
	\end{center}}
\end{table}%

Next we fix $H=2^{-5}$ and decrease $h$, and report the result in Table~\ref{table:fixH_cha1}. We observe that the resolution of the defect has more pronounced influence on the error $e(u_0)$.
\begin{table}[tbhp]
	{\footnotesize
		\caption{The error of Example 1 in the channel defect with $H=2^{-5}$.}
		\label{table:fixH_cha1}
		\begin{center}
			\begin{tabular}{|c|ccc|}
				\hline
				h&$2^{-7}$&$2^{-8}$&$2^{-9}$\\
				\hline
				$e(u^\eps)$& 3.69e-1& 1.99e-1& 7.73e-2\\
				rate&&0.90&1.36\\
				\hline
				$e(u_0)$& 5.10e-2& 4.90e-2& 4.86e-2 \\
				\hline
			\end{tabular}
	\end{center}}
\end{table}%

We report the results for \textbf{Example 2} in Table~\ref{table:fixh_cha2} and Fig.~\ref{fig:channel_noscale}. The method still works with reasonable accuracy. However, from Fig.~\ref{fig:channel_noscale}, we find that the error $e(u^\eps)$ is worse than that in \textbf{Example 1}, which may be due to the poor regularity of the solution inside the defect.
\begin{table}[tbhp]
	{\footnotesize
		\caption{The error of Example 2 in the channel defect with $h=2^{-9}$.}
		\label{table:fixh_cha2}
		\begin{center}
			\begin{tabular}{|c|ccc|}
				\hline
				H&$2^{-4}$&$2^{-5}$&$2^{-6}$\\
				\hline
				$e(u_0)$& 8.85E-2& 4.42E-2& 2.17E-2\\
				rate&&1.00&1.03\\
				\hline
				$e(u^\eps)$& 5.73e-2& 4.80e-2& 4.03e-2\\
				\hline
			\end{tabular}
	\end{center}}
\end{table}%

\begin{figure}[tbhp]
	\centering
	{\includegraphics[width=6cm]{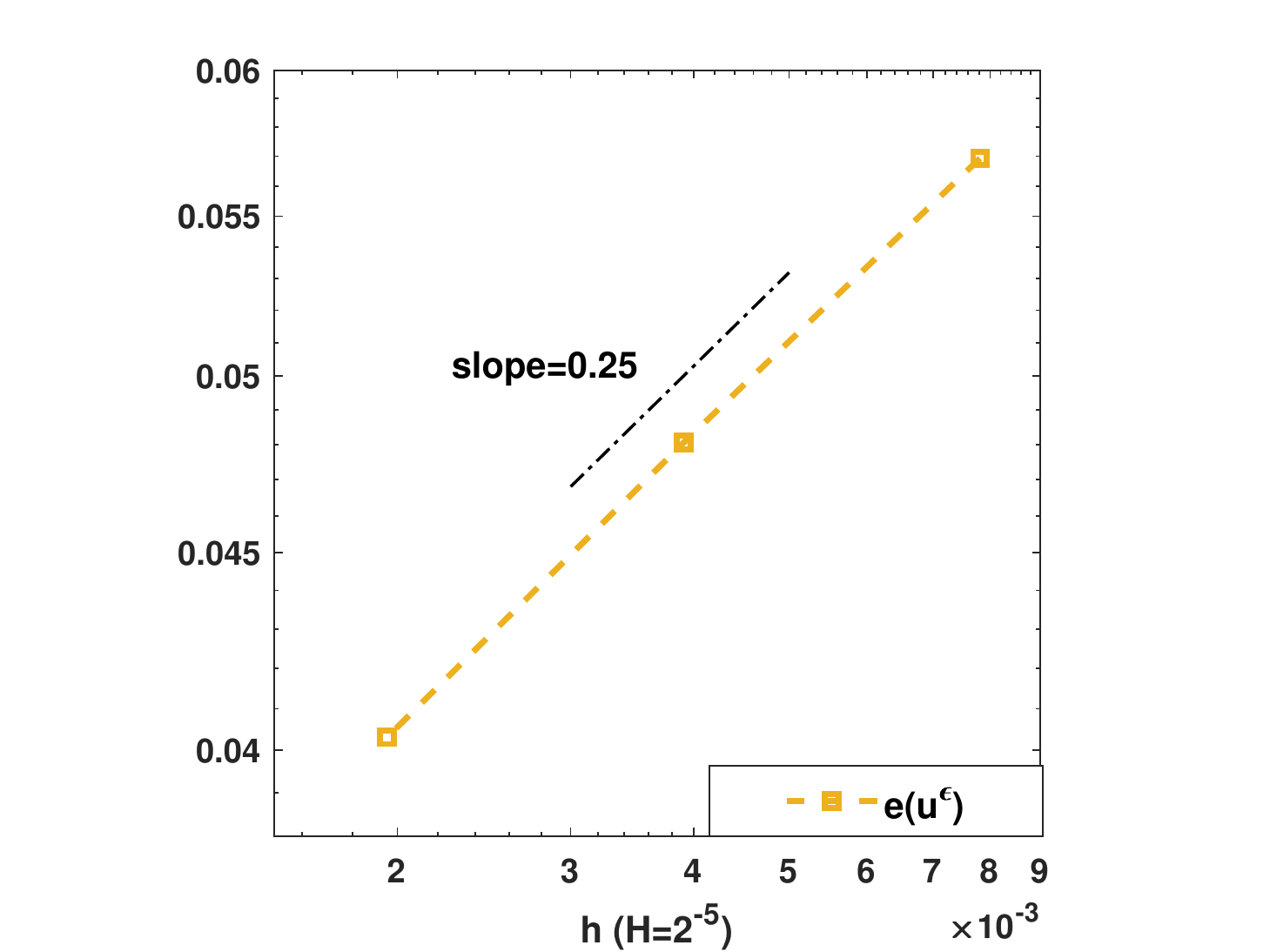}}
	\caption{The error of Example 2 in the channel defect with a fixed $H=2^{-5}$.}
	\label{fig:channel_noscale}
\end{figure}
\subsubsection{Ellipse defects}
We choose two slender ellipse defects as $K_0$ with major axis $0.5$ and minor axis $0.02$; See Fig.~\ref{fig:localell}.  $K_1$ is a rectangle of size $0.54\times0.06$ that contains the ellipse; see Fig.~\ref{fig:rho}$_c$.
\begin{figure}[tbhp]
	\centering
	\includegraphics[width=5.8cm]{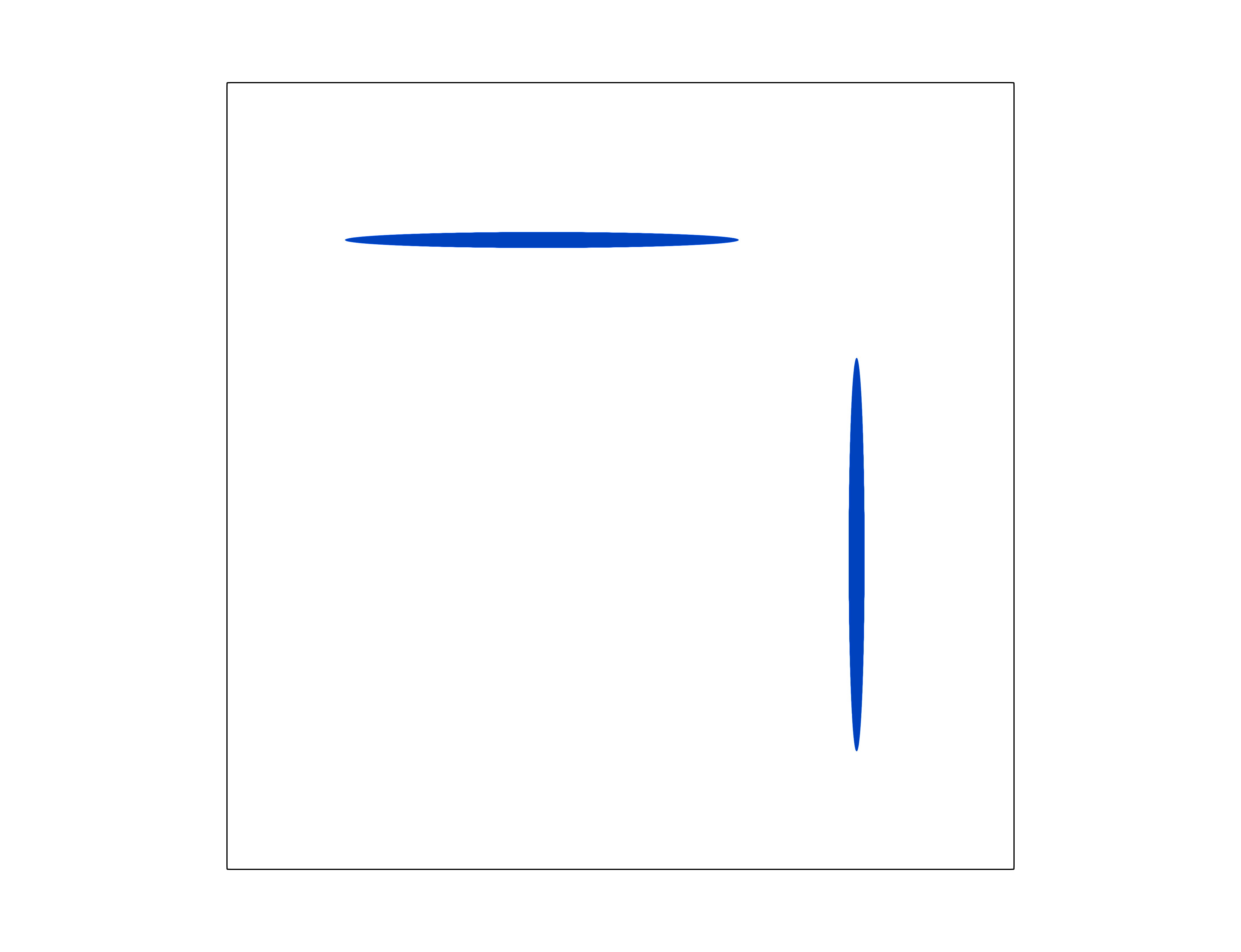}
	\caption{Ellipse defects}
	\label{fig:localell}
\end{figure}

We plot the relative errors in Fig.~\ref{fig:ellipse_H1}$_a$ and Fig.~\ref{fig:ellipse_H1}$_b$ with a fixed ratio $H/h$. The method works for both examples. If we refine the mesh in both subdomains simultaneously, then the energy error is around the first order.
\begin{figure}[tbhp]
	\centering
	{\includegraphics[width=5.8cm]{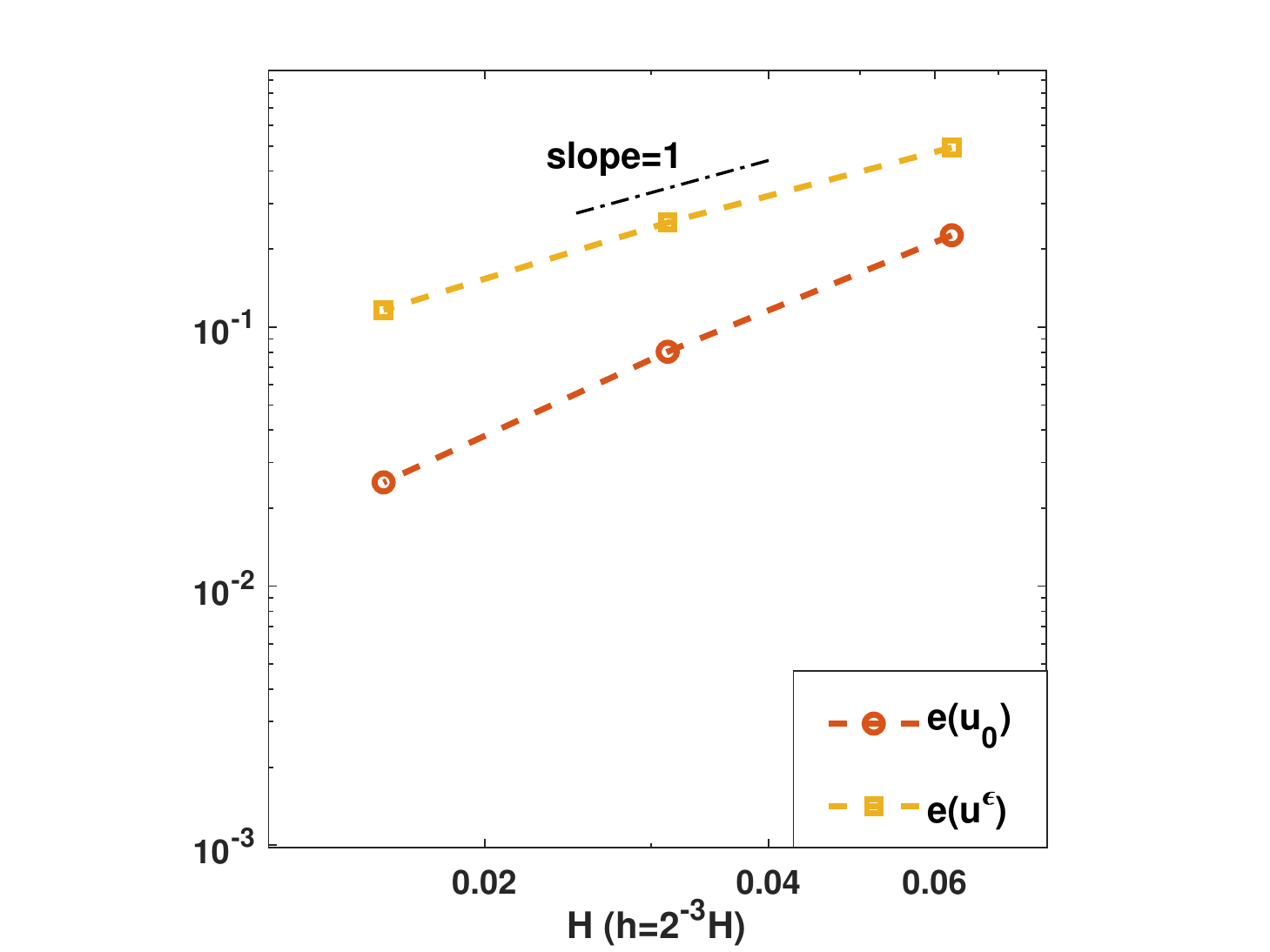}}
	{\includegraphics[width=5.8cm]{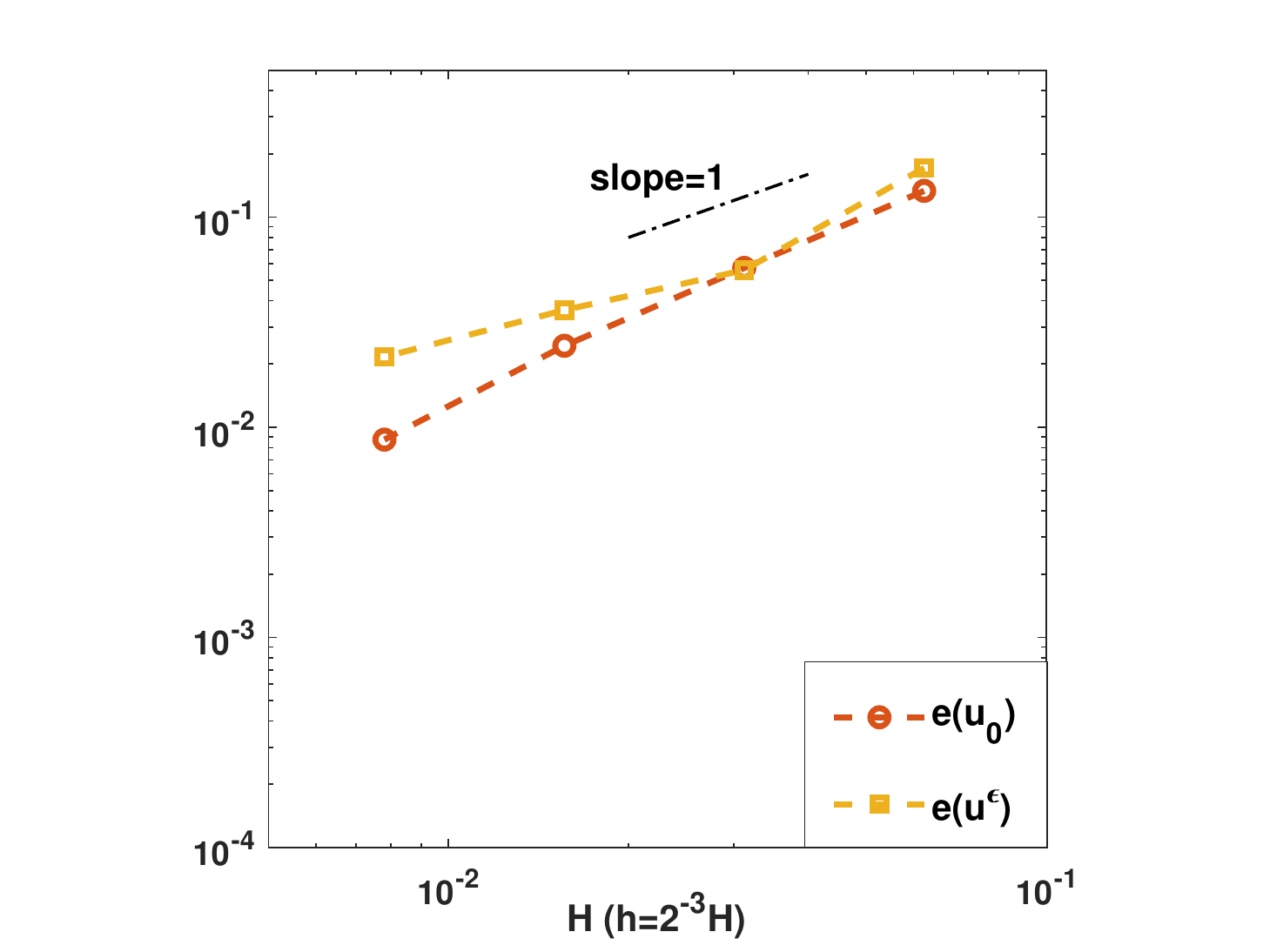}}
	\caption{The relative error for the ellipse defects.}
	\label{fig:ellipse_H1}
\end{figure}
\subsection{Comparison with the original hybrid method}
In the last test, we compare the present method with the hybrid method with a body fitted mesh~\cite{HuangLuMing:2018}. Two kinds of mesh are plot in Fig.~\ref{fig:mesh} for an illustration. Let $K_0=(0.5,0.5)+(-L,L)^2$  and $K_1=(0.5,0.5)+(-L-\delta,L+\delta)^2$ with $L=0.05$ and $\delta=0.01$.
\begin{figure}[tbhp]
	\centering
	{\includegraphics[width=5.8cm]{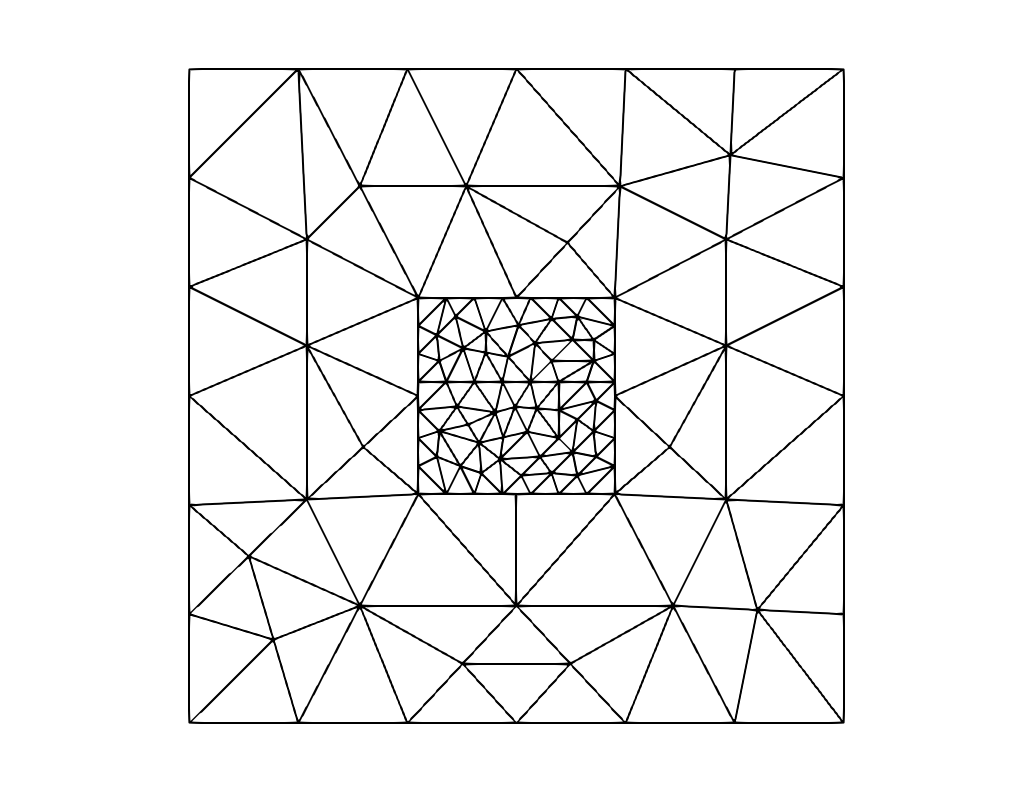}}
	{\includegraphics[width=5.8cm]{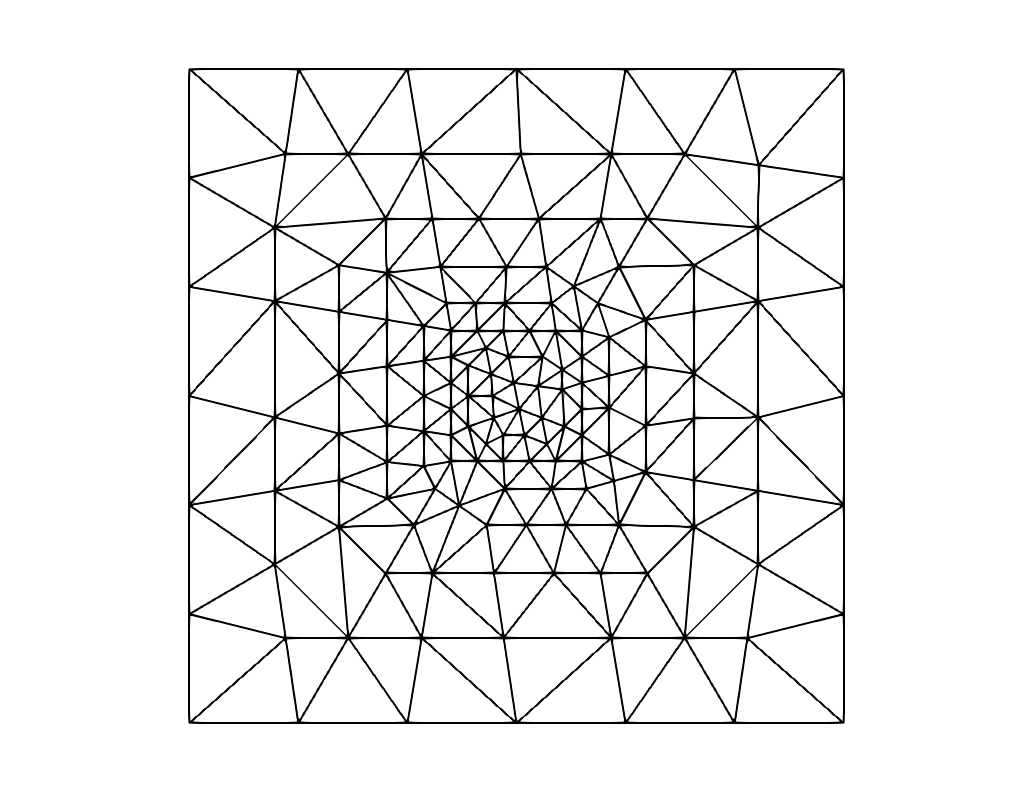}}
	\caption{Left: non-matching grid; Right: matching grid.}
	\label{fig:mesh}
\end{figure}

We choose a mesh with $N=28585$ degrees of freedom for the Nitsche hybrid method, and $N=60564$ degrees of freedom for the original hybrid method. The results for \textbf{Example 1} are summarized in Table~\ref{table:compare}. It seems that the accuracy of both methods are comparable, while the total degrees of freedom of the Nitsche hybrid method is less than one half of the original hybrid method.
\begin{table}[tbhp]
	{\footnotesize
		\caption{Comparison between the Nitsche hybrid method and the original hybrid method.}
		\label{table:compare}
		\begin{center}
			\begin{tabular}{|c|cc|}
				\hline
				& Nitsche hybrid method &hybrid method\\
				\hline
				DOF &28585&60564\\
				$e(u^\eps)$&4.60e-2&4.64e-2\\
				$e(u_0)$&9.30e-3&5.66e-3 \\
				\hline
			\end{tabular}
	\end{center}}
\end{table}
\section{Conclusion}
We present a hybrid method that captures the macroscopical and microscopical information simultaneously in the framework of the Nitsche's variational formulation. A general approach for construction the transition function is proposed. This method admits non-matching grids and works for defects with irregular shape. Hence the method is more efficient and flexible than the original hybrid method~\cite{HuangLuMing:2018}. We prove that the method converges for problems with bounded and measurable coefficients. Rate of convergence has been derived for the periodic media and almost-periodic media.
A possible extension of the present work is to deal with more realistic problem such as parabolic problems with time varying boundary conditions, which is allowed by the Nitsche's variational formulation~\cite{Quar:2016}. We shall leave this for further pursuit.
\section*{Acknowledgments}
The authors would like to thank Professor Jianfeng Lu and Dr Yufang Huang for the discussion on the topic in the earlier stage of the present work. We also thank the anonymous referee for valuable comments.
\appendix

	\section{Proof of Lemma~\ref{lemma:enrich}}
	We define an extension operator: for any $v\in X_{h,H}$
	\[
	E_hv(p){:}=\left\{\begin{aligned}
		v(p)\quad &\text{if\quad} p\in(\N_H\cup \N_h)\setminus \N_\Gamma,\\
		\omega_1v_1(p)+\omega_2v_2(p)\quad &\text{if\quad}p\in \N_\Gamma^0,
	\end{aligned}\right.
	\]
where $v=(v_1,v_2)\in X_h\x X_H$. For each $p\in\N_\Gamma\setminus\N_\Gamma^0$,  there exists $\tau\in\T_H$ such that $p$ sits on the boundary of $\tau$ by \textbf{Assumption B}. Noting that $E_h v$ is well-defined over $\tau$, we may define
\(
	E_h v(p){:}=(E_h v)_{\tau}(p).
\)
The definition of $E_h$ and the uniqueness of the Lagrange interpolation over interface mesh $\E_\cap$ ensure that $E_h v$ is continuous across $\Gamma$. Therefore, $E_hv\in X_{h,H}\cap H^1(D)$ and $E_hv\equiv v$ in $D\setminus\T_{h,H}^\Gamma$.
	
Over each $\tau\in\T_H^\Gamma$, we write
\begin{align*}
		v-E_hv&=\sum_{p\in\N(\tau)\cap\N_\Gamma^0}(v_2-E_hv)(p)\phi_{\tau, p}\\
		&=\sum_{p\in\N(\tau)\cap\N_\Gamma^0}\omega_1(v_2-v_1)(p)\phi_{\tau, p},
	\end{align*}
	where $\phi_{\tau,p}$ is the basis function on $\tau$ associated with node $p$. A  scaling argument shows that, there exist $c_n$ and  $C_n$ independent of $h_\tau$ such that
	\begin{equation}\label{eq:basisest}
		c_nh_\tau^{n-2m}\le\abs{\phi_{\tau, p}}_{m, \tau}^2\le C_nh_\tau^{n-2m}.
	\end{equation}
	Combining the elements in $\T_{H}^\Gamma$ and using the standard inverse inequality~\cite{Ciarlet:1978}, we obtain
	\begin{align}
		\sum_{\tau\in\T_{H}^\Gamma}h_\tau^{2m-2}\abs{v-E_h}_{m,\tau}^2
		&\le C \sum_{e\in \E_\cap}\dfrac{h_e^2}{(h_e+H_e)^2}H_e^{n-2}\nm{\jump{v}}{0,\infty,e}^2\nn\\
		&\le C\sum_{e\in \E_\cap}\dfrac{h_e^2}{(h_e+H_e)^2}H_e^{n-2}h_e^{-d+1}\nm{\jump{v}}{0,e}^2\nn\\
		&\le C\sum_{e\in \E_\cap}h_e^{-1}\nm{\jump{v}}{0,e}^2,\label{eq:approxH}
	\end{align}
	where we have used
	\[
	h_e^{-n+3}H_e^{n-2}/(h_e+H_e)^2\le(h_e/H_e)^{-n+4}h_e^{-1}
	\le h_e^{-1}\qquad n=2,3.
	\]
	
	Next, over each $\tau\in\T_h^\Gamma$, we write
	\begin{align*}
		v-E_hv&=\sum_{p\in\N(\tau)\cap\Gamma}(v_1-E_hv)(p)\phi_{\tau, p}\\
		&=\sum_{p\in\N(\tau)\cap\Gamma}(v_1-v_2)(p)\phi_{\tau,p}+\sum_{p\in\N(\tau)\cap\Gamma}(v_2-E_hv)(p)\phi_{\tau, p}
\end{align*}
	
By \textbf{Assumption B}, there exists $\wt{\tau}\in\T_H$ such that $p$ sits on the boundary of $\wt{\tau}$, and
\begin{align*}
E_hv(p)-v_2(p)&=\sum_{\wt{p}\in\N({\wt{\tau}})\cap\N_\Gamma^0}(E_hv-v_2)(\wt{p})
\phi_{\wt{\tau},\wt{p}}(p)\\		&=\sum_{\wt{p}\in\N({\wt{\tau}})\cap\N_\Gamma^0}\omega_1(v_1-v_2)(\wt{p})\phi_{\wt{\tau},\wt{p}}(p).
\end{align*}
	A scaling argument shows that, there exist $\wt{c}_n$ and  $\wt{C}_n$ independent of $h_{\tau}$ such that
	\[
	\wt{c}_nh_{\tau}^{n-2m}\le\abs{\phi_{\tau, p}}_{m,\tau}^2\le\wt{C}_dh_{\tau}^{n-2m}.
	\]
	By \textbf{Assumption A}, the number of the hanging nodes on each $e\in\E_H$ may be bounded by $c(H_e/h_\Gamma)^{n-1}$.  Thus, proceeding along the same line that leads to~\eqref{eq:approxH} and using the above estimate and the inverse inequality, we obtain
	\begin{align}
		\sum_{\tau\in\T_{h}^\Gamma}h_\tau^{2m-2}\abs{v-E_hv}_{m,\tau}^2
		\le& C\sum_{e\in\E_\cap}h_e^{n-2}h_e^{-n+1}\nm{\jump{v}}{0,e}^2\nn\\
		&+C\sum_{e\in \E_\cap}({H_e}/{h_\Gamma})^{n-1}h_e^{n-2}h_e^{-n+1}\omega_1^2\nm{\jump{v}}{0,e}^2\nn\\
		\le& C\sum_{e\in\E_\cap}h_e^{-1}\nm{\jump{v}}{0,e}^2,\label{eq:approxh}
	\end{align}
	where we have used
	\[
	({H_e}/{h_\Gamma})^{n-1}h_e^{-1}\omega_1^2\le (H_e/h_e)^{n-3}h_e^{-1}\le h_e^{-1}\qquad n=2,3.
	\]
	
	Combining~\eqref{eq:approxH} and~\eqref{eq:approxh}, we obtain~\eqref{eq:approx1}.

	\section{Proof of~\eqref{eq:l2part1}}
	For any $\chi\in X_{h,H}^0$, using ~\eqref{eq:dual} and the enriching operator $E_h$ defined in Lemma~\ref{lemma:enrich}, we obtain
	\begin{align*}
		\dual{g}{u_0-\wt{u}}&=\dual{g}{u_0-E_h\wt{u}}+\dual{g}{E_h\wt{u}-\wt{u}}\\
		&=\dual{\A\na(u_0-E_h\wt{u})}{\na \psi_g}+\dual{g}{E_h\wt{u}-\wt{u}}\\
		&=\dual{f}{\psi_g}-\sum_{i=1}^2\int_{K_i}\A\na\wt{u}\na\psi_g\dx+\sum_{i=1}^2\int_{K_i}\A\na(\wt{u}-E_h\wt{u})\na\psi_g\dx\\
		&\quad+\dual{g}{E_h\wt{u}-\wt{u}},
	\end{align*}
	which may be further expanded as
	\begin{align*}
		\dual{g}{u_0-\wt{u}}
		&=\dual{f}{\chi}-\sum_{i=1}^2\int_{K_i}\A\na\wt{u}\na\chi\dx+\dual{f}{\psi_g-\chi}-\sum_{i=1}^2\int_{K_i}\A\na\wt{u}\na(\psi_g-\chi)\dx\\
		&\quad+\dual{g}{E_h\wt{u}-\wt{u}}-\sum_{i=1}^2\int_{K_i}\A\na(E_h\wt{u}-\wt{u})\na\chi\dx\\
		&\quad+\sum_{i=1}^2\int_{K_i}\A\na(\wt{u}-E_h\wt{u})\na(\psi_g-\chi)\dx.
	\end{align*}
	
	Using~\eqref{eq:hodisvar},~\eqref{eq:magic} and integration by parts, for any
	piecewise constant matrix $\bar{\A}$ over $\T_{h,H}$, we have
	\begin{align*}
		&\dual{g}{u_0-\wt{u}}\\
		=&\left\{\dual{f+\na\cdot\bar{\A}\na\wt{u}}{\psi_g-\chi}-\sum_{e\in\E_\cap}\int_e\jump{\bar{\A}\na\wt{u}\cdot\n}\aver{\psi_g-\chi}^\omega\ds\right\}\\
		&+\left\{\dual{g+\na\cdot\bar{\A}^T\na\chi}{E_h\wt{u}-\wt{u}}-\sum_{e\in\E_\cap}\int_e\jump{\bar{\A}^T\na\chi\cdot\n}\aver{E_h\wt{u}-\wt{u}}^\omega\ds\right\}\\
		&+\left\{\sum_{i=1}^2\int_{K_i}\A\na(\wt{u}-E_h\wt{u})\na(\psi_g-\chi)\dx+\sum_{e\in\E_\cap}\int_e\frac{\gamma}{H_e+h_e}\jump{\wt{u}}\jump{\chi}\ds\right\}\\
		&+\left\{-\sum_{i=1}^2\int_{K_i}(\A-\bar{\A})\na\wt{u}\na(\psi_g-\chi)\dx-\sum_{e\in\E_\cap}\int_e\aver{(\A-\bar{\A})\na \wt{u}}_\omega\jump{\chi}\ds\right\}\\
		&+\left\{-\sum_{i=1}^2\int_{K_i}(\A-\bar{\A})\na(E_h\wt{u}-\wt{u})\na\chi\dx-\sum_{e\in\E_\cap}\int_e\aver{(\A-\bar{\A})\na\chi}_\omega\jump{\wt{u}}\ds\right\}\\
		=&I_1+\cdots+I_5.
	\end{align*}
	
	Using Lemma~\ref{lemma:post}, we obtain
	\begin{equation}\label{eq:L2_1st}
		\begin{split}
			\abs{I_1}
			&\le C\left(\wnm{u_0-\wt{u}}+\text{Osc}(f)+\max_{\tau\in {\T_{h,H}}}\nm{\A-\bar{\A}}{0,\infty,\tau}\nm{f}{-1,D}\right)\\
			&\quad\times\sum_{\tau\in \mc{T}_{h,H}}\Lr{h_\tau^{-1}\nm{\psi_g-\chi}{0,\tau}+\abs{\psi_g-\chi}_{1,\tau}}.
		\end{split}
	\end{equation}
	
	Using Lemma~\ref{lemma:post} and  the enriching estimates~\eqref{eq:approx0},~\eqref{eq:enrich-trace}, we obtain
	\[
	\abs{I_2}\le C\varrho\left(\wnm{\psi_g-\chi}+\textsc{osc}(g)+\max_{\tau\in {\T_{h,H}}}\nm{\A-\bar{\A}}{0,\infty,\tau}\nm{g}{-1,D}\right)\wnm{u_0-\wt{u}},
	\]
	where we have used the fact $\jump{u_0}\equiv 0$ on $\Gamma$.
	
	Using ~\eqref{eq:approx0} and the fact $\jump{\psi_g}\equiv\jump{u_0}\equiv 0$ on $\Gamma$, we have 		
	\[
	\abs{I_3}\le C\varrho\wnm{\wt{u}-u_0}\wnm{\psi_g-\chi}.
	\]
	Finally,
	\begin{align*}
		\abs{I_4}&\le C \max_{\tau\in {\T_{h,H}}}\nm{\A-\bar{\A}}{0,\infty,\tau}\Lr{\sum_{i=1}^2\abs{\wt{u}}_{1,K_i}^2}^{1/2}\wnm{\psi_g-\chi}\\
		&\le C \max_{\tau\in {\T_{h,H}}}\nm{\A-\bar{\A}}{0,\infty,\tau}\Lr{\wnm{u_0-\wt{u}}+\nm{f}{-1,D}}\wnm{\psi_g-\chi},
	\end{align*}
	and
	\[
	\abs{I_5}\le C\varrho \max_{\tau\in {\T_{h,H}}}\nm{\A-\bar{\A}}{0,\infty,\tau}\Lr{\wnm{\psi_g-\chi}+\nm{g}{-1,D}}\wnm{u_0-\wt{u}}.
	\]
	Combining all the above estimates, we obtain~\eqref{eq:l2part1}.
	%

%

%
%

\bibliographystyle{siamplain}
\bibliography{HybridMethod-2}

\end{document}